\documentclass[12pt]{amsart}
\usepackage{amssymb,amsmath,amsthm,mathrsfs,ulem}

\newtheorem{thm}{Theorem}[subsection]
\newtheorem{theorem}[thm]{Theorem}
\newtheorem{lemma}[thm]{Lemma}
\newtheorem{corollary}[thm]{Corollary}
\newtheorem{proposition}[thm]{Proposition}
\numberwithin{equation}{section} \theoremstyle{definition}

\newtheorem*{remark}{Remark}
\newtheorem*{remarks}{Remarks}
\usepackage{graphicx,color}
\usepackage{amsmath}
\usepackage{latexsym}
\usepackage{epsfig}
\usepackage{amsfonts}
\usepackage{amssymb}
\usepackage[margin=2cm]{geometry}
\newcommand{\dig}{{\rm diag}}
\newcommand{\diag}{{\rm diag}}
\newcommand{\Tr}{{\rm Tr}}
\newcommand{\B}{\mathcal{B}}
\newcommand{\U}{\mathcal{U}}
\newcommand{\A}{\mathcal{A}}
\newcommand{\I}{\mathcal{I}}
\newcommand{\oo}{\mathcal{O}_F}
\newcommand{\m}{^{\text{-1}}}
\newcommand{\ord}{{\rm ord}}

\newcommand{\PGL}{{\rm PGL}}
\newcommand{\GL}{{\rm GL}}
\newcommand{\PSL}{{\rm PSL}}

\newcommand{\Z}{\mathbb Z}

\newcommand{\G}{\Gamma}
\newcommand{\g}{\gamma}
\newcommand{\ka}{\kappa}
\newcommand{\ve}{\varepsilon}
\newcommand{\de}{\delta}
\newcommand{\vol}{{\rm vol}}
\begin{document}
\title[Zeta Functions of Complexes Arising from $\PGL(3)$]{Zeta Functions of Complexes Arising from $\PGL(3)$}
\author{Ming-Hsuan Kang and Wen-Ching Winnie Li}
\address{Ming-Hsuan Kang\\ Department of Mathematics\\ Purdue University\\
West Lafayette, IN 47907} \email{\tt kmsming@gmail.com}
\address{Wen-Ching Winnie Li\\ Department of Mathematics\\ The Pennsylvania State University\\
University Park, PA 16802} \email{\tt wli@math.psu.edu}

\thanks{The research of both authors are supported in part by the DARPA
grant HR0011-06-1-0012 and the NSF grants DMS-0457574 and DMS-0801096. Part of the research was performed
while both authors were visiting the National Center for Theoretical
Sciences, Mathematics Division, in Hsinchu, Taiwan. They would like
to thank the Center for its support and hospitality.}


\begin{abstract} In this paper we obtain a closed form expression of
the zeta function $Z(X_\G, u)$ of a finite quotient $X_\G = \G
\backslash \PGL_3(F)/\PGL_3(\oo)$ of the Bruhat-Tits building of
$\PGL_3$ over a nonarchimedean local field $F$. Analogous to a graph
zeta function, $Z(X_\G, u)$ is a rational function and it satisfies
the Riemann hypothesis if and only if $X_\G$ is a Ramanujan complex.
\end{abstract}
\maketitle

\section{Introduction}
First introduced by Ihara \cite{Ih} for groups and later
reformulated by Serre for regular graphs,
the zeta function of a finite,
connected, undirected graph $X$ is defined as
$$ Z(X,u) = \prod_{[C]} (1 - u^{l([C])})^{-1},$$
where the product is over equivalence classes $[C]$ of backtrackless
tailless primitive cycles $C$, and $l([C])$ is the length of a cycle
in $[C]$.\footnote{ A cycle has a starting point and orientation.
Two cycles are equivalent if one is obtained from the other by changing the starting vertex. A cycle is tailless if all cycles equivalent to it are backtrackless; it is primitive if it is not a repetition of a shorter cycle more than once.}
Taking the logarithmic derivative of $Z(X, u)$, one gets
$$Z(X,u) =  \exp \bigg(\sum_{n \ge 1} \frac{N_n}{n} u^n \bigg),$$ where
$N_n$ counts the number of backtrackless and tailless cycles in $X$
of length $n$.

Not only formally analogous to a curve zeta function, the graph zeta
function is also a rational function. This can be seen in two ways.
The first is the result of Ihara:

\begin{theorem}[Ihara \cite{Ih}]
Suppose $X = (V, E)$ with vertex set $V$ and edge set $E$ is
$(q+1)$-regular. Then its zeta function
 is a rational function of the form
$$ Z(X,u) = \frac{(1 - u^2)^{\chi(X)}}{\det(I - Au + qu^2I)},$$
where $\chi(X) = \#(V) - \#(E)$ is the Euler characteristic of $X$
and $A$ is the adjacency matrix of $X$.
\end{theorem}

If $X$ is not regular, the same expression holds with $qI$ replaced
by the valency matrix of $X$ minus the identity matrix. This
was proved by Bass \cite{Ba} and Hashimoto \cite{Ha2}; Stark and
Terras provided several proofs in \cite{ST}, while Hoffman \cite{Ho}
gave a cohomological interpretation. The reader is referred to
\cite{ST} and the references therein for the history and various
zeta functions attached to a graph.

Endow two orientations on each edge of $X$. Define the
neighbors of the directed edge $u \to v$ to be the edges $v \to w$ with $w \ne u$.
The edge adjacency matrix $A_e$ has its rows and columns indexed by
the directed edges $e$ of $X$ such that the $e e'$ entry is 1 if $e'$ is a neighbor
of $e$, and $0$ otherwise.
Hashimoto
\cite{Ha} observed that $N_n = \Tr A_e^n$ so that
$$Z(X, u) = \frac {1}{\det(I - A_e u)}.$$
This gives the second viewpoint of the rationality of the graph zeta
function.

 A $(q+1)$-regular graph $X$ is called {\it Ramanujan} if all
eigenvalues $\lambda$ of its adjacency matrix $A$ other than $\pm
(q+1)$ satisfy $|\lambda| \le 2 \sqrt q$ (cf. \cite{LPS}). The
Ramanujan graphs are optimal expanders with extremal spectral
property. It is easily checked that $X$ is Ramanujan if and only if
its zeta function $Z(X,u)$ satisfies the Riemann hypothesis, that
is, the poles of $Z(X,u)$ other than $\pm 1$ and $\pm q^{-1}$,
called nontrivial poles, all have absolute value $q^{-1/2}$ (cf.
\cite{ST}).

When $q$ is a prime power, the universal cover of a $(q+1)$-regular
graph can be identified with the $(q+1)$-regular tree on
$\PGL_2(F)/\PGL_2(\oo)$ for
 a nonarchimedean local field $F$ with ring of integers $\oo$ and $q$ elements in its
residue field. Let $\pi$ be a uniformizer of $F$. The vertices of
the tree are $\PGL_2(\oo)$-cosets and the directed edges are
$\I$-cosets, where $\I$ is the Iwahori subgroup of $\PGL_2(\oo)$.
Moreover, the (vertex) adjacency operator $A$ on the tree is the
Hecke operator given by the double coset $\PGL_2(\oo)\diag(1, \pi)
\PGL_2(\oo)$ and the edge adjacency operator $A_e$ is the
Iwahori-Hecke operator given by the double coset $\I \diag(1, \pi)
\I$. One obtains a $(q+1)$-regular graph by taking a left quotient
by a torsion-free discrete cocompact subgroup of $\PGL_2(F)$.

This set-up has a higher dimensional extension to the Bruhat-Tits
building $\B_n$ associated to $\PGL_n(F)/\PGL_n(\oo)$, which is a
simply connected $(q+1)$-regular $(n-1)$-dimensional simplicial
complex. Its vertices are $\PGL_n(\oo)$-cosets, naturally partitioned
into $n$ types, marked by $\mathbb Z/n\mathbb Z$.
 There are $n-1$
Hecke operators $A_i$, for $1 \le i \le n-1$,
 associated to $\PGL_n(\oo)$-double cosets represented by
$\dig( 1,..., 1,\pi, ..., \pi)$
 with determinant $\pi^i$. A finite
quotient $X_\G = \G \backslash \B_n$ of $\B_n$ by a torsion-free discrete
cocompact  subgroup $\G$ preserving the types of vertices is
again a $(q+1)$-regular finite complex. It is called a {\it
Ramanujan complex} if all the nontrivial eigenvalues of $A_i$ on
$X_\G$ fall within the spectrum of $A_i$ on the universal cover
$\B_n$. See \cite{Li} for more details. Three explicit constructions
of infinite families of Ramanujan complexes are given in Li
\cite{Li}, Lubotzky-Samuels-Vishne \cite{LSV1} and Sarveniazi
\cite{Sa}, respectively, using deep results on the Ramanujan
conjecture over function fields for automorphic representations of
the multiplicative group of a division algebra by
Laumon-Rapoport-Stuhler \cite{LRS} and of $\GL_n$ by Lafforgue
\cite{La}. Further, the paper \cite{LSV2} discusses what kind of
$\G$ would fail to yield a Ramanujan complex.

To extend the results from graphs to complexes, one seeks a
similarly defined zeta function of closed geodesics in
 $X_\G$ with the following properties:
\begin{itemize}
\item[(1)] it is a rational function with a closed form expression;

\item[(2)] it captures both topological and spectral information of
$X_\G$; and

\item[(3)] it satisfies the Riemann hypothesis if and only if
 $X_\G$ is a Ramanujan complex.
\end{itemize}
Questions of this sort were previously considered in Deitmar
\cite{De1}, \cite{De2}, and Deitmar-Hoffman \cite{DH}, where partial
results were obtained.

The purpose of this paper is to present a zeta function with the
asserted properties for the case $n = 3$. In what follows, we fix a
local field $F$ with $q$ elements in its residue field as before.
Write $G$ for $\PGL_3(F)$, $K$ for its maximal compact subgroup
$\PGL_3(\oo)$, and $\B$ for the Bruhat-Tits building $\B_3$. Similar to a tree,
the geometric objects in the building $\B$ can be parametrized algebraically.
More precisely, the vertices of $\B$ are the right $K$-cosets on which the group $G$
acts transitively by left translation. A directed edge in $\B$ has type $1$ or $2$, given by the type of the ending vertex minus
the type of the initial vertex. Thus opposite edges have different types.
Let $\sigma = \left( \begin{smallmatrix}  & 1 &  \\
& & 1 \\ \pi & & \end{smallmatrix}\right).$
As the stabilizer of the type $1$ edge $K \to \sigma K$, denoted by $e_0$, is
$E := K ~\cap~\sigma K\sigma^{-1}$, the right $E$-cosets parametrize the type $1$ edges of $\B$. The  Iwahori subgroup $ B:= K ~\cap~\sigma K\sigma^{-1}\cap ~\sigma^{-1}K\sigma$
stabilizes the three vertices $K$, $\sigma K$ and $\sigma^2K$ of the chamber $C_0$. Since the stabilizer of
$C_0$ in $G$ is $B \cup B\sigma \cup B \sigma^2$ and the type $1$ edges of $C_0$ are
$\sigma^i e_0$ for $1 \le i \le 3$,  the right
$B$-cosets parametrize the directed chambers $(C, e)$ of $\B$, where
$e$ is a type $1$ edge of the chamber $C$. Define the neighbors of a type $1$ edge
$gK \to g'K$ to be the type $1$ edges $g'K \to g''K$ such that $gK$, $g'K$ and $g''K$ do not form a chamber. Further, the neighbors $(C', e')$ of a directed chamber $(C, e)$ with $e$ the edge $g_1K \to g_2K$ and $ g_3K$ the third vertex of $C$ are defined as follows: $C'$ are the chambers other than $C$ which share
the edge $g_2K \to g_3K$ and the type $1$ edge $e'$ is from $g_3K$ to the third vertex of $C'$.
A geodesic between two vertices in $\B$ is a shortest path in the $1$-skeleton of $\B$. A path in $X_\Gamma$ is called a geodesic if any of its lifting in $\B$ is a geodesic.

The zeta function of $X_\G$ is defined as
$$ Z(X_\G,u) = \prod_{[C]} (1 - u^{l_A([C])})^{-1},$$
\noindent where $[C]$ runs through the equivalence classes of
 tailless primitive closed geodesics consisting of edges of the same type,
 and $l_A([C])$ is the algebraic
length of any geodesic in $[C]$.
\smallskip

{\bf Main Theorem.} {\it Let $\Gamma$ be a discrete cocompact
torsion-free
 subgroup of $G$ such that

(I) $\ord_{\pi} \det \G \subseteq 3\mathbb Z$,  and

(II) $\G$ is regular, namely, the {centralizer} of any non-identity
element of $\G$ in $G$ is a torus.

\noindent Then the zeta function of the $(q+1)$-regular finite
complex $X_\G = \Gamma \backslash \B$ is a rational function
\begin{eqnarray}\label{zeta}
Z(X_\G, u) = \frac{(1-u^3)^{\chi(X_\G)}}{\det(I-A_1u+qA_2u^2-q^3
u^3I)\det(I + L_Bu)},
\end{eqnarray}
in which $\chi(X_\G)$ is the Euler characteristic of $X_\G$, and $L_B$ is the Iwahori-Hecke
operator given by the $B$-double coset $Bt_2\sigma^2 B$, where
$t_2= \left( \begin{smallmatrix}  &  & \pi^{-1} \\
 &1 &  \\\pi & & \end{smallmatrix}\right)$.}

Similar to the graph zeta function, our complex zeta function can be
expressed as
$$ Z(X_\G, u) = \frac{1}{\det(I - L_E u) \det(I - (L_E)^t u^2)} =
\frac{1}{\det(I - L_E u) \det(I - L_E u^2)},$$ where $L_E$ is the
operator given by the double coset $E(t_2 \sigma^2)^2E$, which is also the
adjacency matrix of type $1$ edges in $X_\G$.

Since the opposite of the type $1$ edges are the type $2$ edges, the transpose
$(L_E)^t$ is the adjacency matrix for type $2$ edges. Likewise,
$L_B$ may be viewed as the adjacency matrix of directed chambers in $X_\G$.
Consequently, the identity (1.1) can be expressed in terms of
operators on $X_\G$ as
\begin{eqnarray}\label{zetaidentity}
\frac{(1-u^3)^{\chi(X_\G)}}{\det(I-A_1u+qA_2u^2-q^3 u^3I)} =
\frac{\det(I + L_Bu)}{\det(I - L_E u) \det(I - (L_E)^t u^2)},
\end{eqnarray}
while the parallel identity of operators on a $(q+1)$-regular graph
$X$
 reads
$$\frac{(1 - u^2)^{\chi(X)}}{\det(I - Au + qu^2I)} = \frac{1}{\det(I
- A_eu)}.$$ The similarity is reminiscent of the zeta functions
attached to a surface and a curve over a finite field. Since
(\ref{zetaidentity}) is expressed in terms of the operators on the
finite complex, it is likely to be the prototype of complex zeta
functions in general.

$Z(X_\G, u)$ clearly has properties (1) and (2). Now we discuss its
connection with the Riemann hypothesis. The trivial zeros of
$\det(I-A_1u+qA_2u^2-q^3 u^3I)$ arise from the trivial eigenvalues
of $A_1$ and $A_2$ on $X_\G$; they are $1, q^{-1}, q^{-2}$ and their
multiples by cubic roots of unity. An equivalent statement for
$X_\G$ being Ramanujan is that the nontrivial zeros of
$\det(I-A_1u+qA_2u^2-q^3 u^3I)$ all have absolute value $q^{-1}$
(cf.\cite{Li}), which is the Riemann hypothesis for $Z(X_\G, u)$.

The zeros of each determinant in (\ref{zetaidentity}) are computed in \cite{KLW},
where equivalent statements are obtained.

\begin{theorem}[\cite{KLW}, Theorem 2] The following four statements on $X_\G$ are equivalent.
\begin{enumerate}
\item[(1)] $X_\G$ is a Ramanujan complex;

\item[(2)] The nontrivial zeros of $\det(I-A_1u+qA_2u^2-q^3 u^3I)$ have absolute value $q^{-1}$;

\item[(3)] The nontrivial zeros of $\det(I + L_Bu)$ have absolute values $1$, $q^{-1/2}$ and $q^{-1/4}$; and

\item[(4)] The nontrivial zeros of $\det(I - L_E u)$ have absolute values $q^{-1}$ and $q^{-1/2}$.
\end{enumerate}
\end{theorem}
\noindent Thus the Riemann hypothesis for $Z(X_\G, u)$ is actually a statement concerning the nontrivial zeros of each determinant in (\ref{zetaidentity}),
analogous to the Riemann hypothesis for a surface zeta function.
A representation-theoretical proof of (\ref{zetaidentity}) is given in \cite{KLW}. It should be pointed out that the right hand side of
(\ref{zetaidentity}) is equal to $Z(X_\G, u)/Z_2(X_\G, -u)$,  where $Z_2(X_\G, u)$ is
the zeta function of tailless type $1$ closed galleries in $X_\G$.
As shown in \S10, this quotient also affords another
interpretation as the product of a geometric and an algebraic zeta functions: $Z_1(X_\G, u) Z_-(\G, u)$,
where $Z_1(X_\G, u)= 1/\det(I - L_Eu)$ involves type $1$ geodesic cycles in $X_\G$,
and $ Z_-(\G, u)$ involves conjugacy classes in $\G$ of negative type.
This interpretation gives an infinite product expression of the left hand side of (\ref{zetaidentity})
(cf. Theorem \ref{newzetaidentity}).

This paper is organized as follows. In \S2 the types and lengths of elements in $G$
and geodesics in $\B$ are introduced. Properties of elements in $\G$ and basic concepts
of cycles in the finite complex $X_\G$ are discussed in \S3, while recursive relations of Hecke
operators on $X_\G$ are
laid out in \S4. The  vertex-based homotopy classes of closed geodesics in
$X_\G$ are partitioned into sets indexed by the conjugacy classes
$[\g]$ of $\G$, with each set consisting of  vertex-based homotopy classes
which are base-point free homotopic to the path from $K$ to $\g K$.
 Each set $[\g]$ has a type, algebraic length and
geometric length, defined in terms of those of the rational form
of $\g$, which depends on $\g$ up to conjugacy.
 Theorem \ref{charlAandlG} says that the
lengths of the set $[\g]$ are the minimal respective lengths of the
homotopy cycles contained in the set. Cycles achieving minimal
geometric
 (resp. algebraic) length in each $[\g]$ are called {\it tailless} (resp. {\it
algebraically tailless}). In other words, among the cycles
base-point free homotopic to each other, the shortest ones are
called tailless. This definition also applies to graphs. Algebraically tailless cycles afford an explicit algebraic
characterization, as shown in \S5 and \S6 according as
 $\g$ is split or rank-one split, and hence are more amenable to
computation. We shall see in \S5 and \S6 that, for type $1$ and type
$2$ cycles, there is no distinction between algebraic tailless and
tailless (Corollaries  \ref{typeoftailless} and \ref{rankonetailless}).

While the zeta function only concerns tailless cycles of types $1$
and $2$, to find its closed form, we have to consider all cycles up
to homotopy. Indeed, we shall compute the number of cycles, as well
as those of type $1$, in a set $[\g]$ with given algebraic length.
This is carried out in \S5 and \S6. As shown in \S9, where the Main
Theorem is proved, these numbers can be put together to show that
the logarithmic derivative of the left hand side of
(\ref{zetaidentity}) counts the number of type $1$ tailless closed
geodesics in $X_\G$, namely, those from the logarithmic derivative
of $1/\det(I - L_E u)$, and some extra terms arising from sets
represented by rank-one split $\g$'s.

\S7 and \S8 are devoted to explaining these extra terms. In \S7
we discuss type $1$ tailless closed galleries and define chamber
zeta function $Z_2(X_\G, u)$, while the zeta function on type $1$
tailless closed geodesics, $Z_1(X_\G, u)$, is discussed in \S8. The
boundary of a type $1$ tailless closed gallery is analyzed in \S8.2,
where it is shown that the boundary of an even/odd length gallery
consists of two/one tailless type $1$ cycle(s). The information on
the boundary further leads to a criterion on the chambers occurring
in a type $1$ tailless closed gallery. This in turn allows us to
compute the logarithmic derivative of $\frac{\det(I + L_Bu)}{\det(I
- (L_E)^t u^2)} = \frac{Z_1(X_\G, u^2)}{Z_2(X_\G, -u)}$, which gives
the extra terms.

In \S10 the Ihara (group) zeta function $Z(\G, u)$ attached to $\G$ is introduced,
analogous to the original definition in \cite{Ih} for the case of $\PGL_2(F)$, as
an infinite product over primitive conjugacy classes in $\G$. By separating these
conjugacy classes into positive and negative types, we show that the product
over those with negative type, denoted by $Z_-(\G, u)$, accounts for the
extra terms alluded above, and thus provides a different interpretation of (\ref{zetaidentity}).
Finally we remark that for
$\PGL_2$, Ihara group zeta function coincides with the graph zeta function attached
to the quotient of the tree by the group, but this is no longer true for $\PGL_3$.

\section{Edges and Geodesics in $\B$}

\subsection{Hecke operators}
The group $G$ is the disjoint union of the $K$-double cosets $$
T_{n,m}=K~\diag(1, \pi^m, \pi^{m+n})K$$ as $m, n$ run through all
non-negative integers. We shall also regard $T_{n,m}$ as the Hecke
operator acting on functions $f \in L^2(G/K)$ via
$$T_{n,m}f(gK) = \sum_{\alpha K \in T_{n,m}/K} f (g\alpha K).$$
In particular,
$$ A_1 = T_{1,0} \qquad {\rm and} \qquad A_2 = T_{0,1}.$$

\subsection{Description of type 1 and type 2 edges} The vertices of $\B$ are parametrized by $G/K$.
Given a vertex $gK$, represent $g \in G$ by an element $\tilde g \in \GL_3(F)$. Denote by $L$ the rank three $\oo$-lattice generated by the three column vectors of $\tilde g$. The equivalence class of $L$ depends only on $gK$, and will be identified with $gK$. 
 Two vertices $gK$ and $g'K$ are adjacent if they can be represented by lattices $L$ and $L'$, respectively, such that $\pi L  \subset L'  \subset L$. Three mutually adjacent vertices form a $2$-dimensional simplex, called a  chamber. This structure makes the building $\B$ a simply connected $2$-dimensional simplicial complex.

A vertex $gK$ has a type $\tau(gK)$ defined by $\ord_{\pi} \det g \mod 3$.
Adjacent vertices do not have the same type.
The type of a directed
edge $gK \rightarrow g'K$ is $\tau(g'K) - \tau(gK) = i$, which is
$1$ or $2$. Out of each vertex there are $q^2+q+1$ edges of a given
type. The type $1$ edges out of $gK$ have terminal vertices $g\alpha
K$, where $\alpha K$ are the $K$-cosets contained in the double
coset
$$ A_1 = T_{1,0} = K \left(\begin{matrix} 1 &  &  \\ & 1&  \\
& & \pi \end{matrix}\right)K
= \bigcup_{a, b \in \oo / \pi \oo}\left(\begin{matrix} \pi & a &  b \\ & 1& \\
& & 1\end{matrix}\right)K \bigcup_{c \in \oo / \pi \oo}
 \left(\begin{matrix} 1 &  &  \\ & \pi& c \\
& & 1\end{matrix}\right)K \bigcup  \left(\begin{matrix} 1 &  &  \\ & 1&  \\
& & \pi \end{matrix}\right)K.
$$

\noindent Similarly, the terminal vertices of type 2 edges out of
$gK$ can be described using the following $q^2+q+1$ left K-coset
representatives of $A_2=T_{0,1}$:
$$\left(\begin{matrix} \pi &  &  b \\ & \pi & c\\
& & 1\end{matrix}\right), \left(\begin{matrix} \pi & a &  \\ & 1 &  \\
& & \pi \end{matrix}\right) ~{\rm and} ~\left(\begin{matrix} 1 &  &  \\ & \pi &  \\
& & \pi \end{matrix}\right), ~{\rm where} ~a, b, c \in \oo /\pi
\oo.$$

\subsection{Geodesics and lengths in $\B$}

Since $\B$ is simply connected, all paths between two vertices are
homotopic. By a geodesic between two vertices of $\B$ we mean a path
with shortest length in the 1-skeleton of the building
$\B$, which
is the (undirected) graph with vertex set $G/K$ and adjacency operator
$A_1 + A_2$.

It can be shown that all geodesics between two vertices $g_1K$ and
$g_2K$ with $g_1^{-1}g_2 \in T_{n,m}$ lie in the same apartment, and
they use $n$ type $1$ edges and $m$ type $2$ edges. We say that they
have type $(n, m)$. When $m = 0$ (resp. $n = 0$), the path is called
{\it type $1$} (resp. {\it type $2$}) for short. Given
a geodesic from $g_1K$ to $g_2K$, define $l_G(g_1\m g_2):=n+m$ to be
its {\it geometric
length} and $l_A(g_1\m g_2) := n+2m$ its {\it algebraic length}.
 Note that the same path
traveled backwards has algebraic length $m + 2n$. Further, when the
path has type $1$ or $2$, there is only one geodesic between the two
vertices.

\section{Finite quotients of $\B$}

\subsection{The group $\G$}

Let $\Gamma$ be a discrete cocompact torsion-free subgroup of $G$
which acts on $\B$ by left translations. Then $\G$ intersects any
compact subgroup of $G$ trivially. In particular, $\G$ acts on $\B$
free of fixed points. Denote by $X_{\Gamma} = \Gamma\backslash \B$
the finite quotient, whose vertices are the double cosets $\Gamma
\backslash G /K$. See \cite[\S3]{Sa} for some examples of such $\G$.

\subsection{Homotopy classes of closed paths in $X_\G$}\label{parametrization}

The 1-skeleton of $X_\G$ is an undirected graph with the adjacency
operator $A_1 + A_2$. We study cycles on this graph which are
homotopic in $X_\G$. Recall that all cycles
we consider are contained in the $1$-skeleton of $X_\G$.

 A closed geodesic in $X_\G$ starting at the vertex $\G g K$ can be
lifted to a geodesic path in $\B$ starting at $gK$ and ending at $\g gK$ for
some $\g \in \G$. Two such geodesic cycles in $X_\G$ are homotopic
in $X_\G$ if and only if their liftings in $\B$ to two geodesic paths starting
at $gK$ have the same ending vertex. Denote by $\kappa_\g(gK)$ the homotopy class of
the geodesic paths from $gK$ to $\g gK$ in $\B$. When
projected to $X_\G$, these geodesic paths become homotopic geodesic cycles which have
 shortest geometric length
among all cycles in its homotopy class in $X_\G$. By abuse of
notation, we also use $\kappa_\g(gK)$ to denote the homotopy class of its projection
in $X_\G$. Thus the fundamental group of $X_\G$ based at $\G g K$ is
$$\pi_1(X_\G, \G g K) = \{\kappa_\g(gK) : \g \in \G \}.$$

\noindent Since $\G$ has no fixed points, all $\kappa_\g(gK)$ are
distinct and $\pi_1(X_\G, \G g K)$ is isomorphic to $\G$.

When all base points are taken into account, the set of all vertex-based
homotopy classes of all closed geodesics in $X_\G$, that is, $\coprod_{\G gK \in \G \backslash G/K}
\pi_1(X_\G, \Gamma gK)$, can be expressed as $\G \times \G \backslash G/K$.

 For each conjugacy class of $\G$ fix a
representative $\g$ and denote that class by $\langle \g \rangle_{\G}$. Let $[\G]
= \{ \g \}$ be the set of representatives of conjugacy classes.
Since the conjugacy class of $\g$ in $\G$ corresponds bijectively to
$\G$ modulo the centralizer $C_\G(\g)$ of $\g$ in $\G$, so $\G
= \coprod_{\g \in [\G]} \langle \g \rangle_\G$ corresponds bijectively to $\coprod_{\g \in [\G]}
C_\G(\g)\backslash \G$, and consequently $\Gamma \times  \G \backslash G/K$ corresponds bijectively to
$$ \coprod_{\g \in [\G]} (C_\Gamma(\g)\backslash \Gamma) \times ( \Gamma\backslash G/K )=
  \coprod_{\g \in [\G]}  C_\Gamma(\g)\backslash G/K.$$
Letting, for each $\g \in [\G]$,
\begin{eqnarray}\label{classofgamma}
[\g] = \{\kappa_\g (gK) ~|~g \in C_\Gamma(\g)\backslash G/K \},
\end{eqnarray} we can express the set of all vertex-based homotopy classes of $X_\G$
as the disjoint union of $[\g]$ over $\g \in [\G]$.

Two vertex-based homotopy classes of $X_\G$ are said to be {\it base-point
free homotopic} if a cycle in one class is obtained from a cycle in
the other class by repeated applications the following procedures:

(H1) Shifting the starting vertex to another vertex on the cycle;

(H2) Replacing the cycle by a homotopic cycle while holding the end
points fixed.

\noindent A standard topological argument proves the following geometric interpretation of the set $[\g]$.

\begin{proposition}\label{basepointfreehomotopy}
Let $\g \in [\G]$. The set $[\g]$ defined by (\ref{classofgamma})
has the following properties:

(i) It is closed under base-point free homotopy;

(ii) Any two classes in $[\g]$ are base-point free homotopic;

(iii) The set $[\g]$ is independent of the choice of representative $\g$ in the conjugacy class $\langle \g \rangle_\G$.

\noindent Consequently, $[\g]$ consists of all vertex-based homotopy
classes which are base-point free homotopic to $\kappa_\g(K)$.
\end{proposition}

\subsection{Classification of elements in $\G$}
Given $\g \in \Gamma$, the eigenvalues of $\g$ are defined to be those of any
$\tilde \g \in \GL_3(F)$ representing $\g$. So they are well-defined up to common
multiples in $F^\times$. In the next subsection an appropriate representative of
eigenvalues of $\g$ will be chosen to facilitate future discussions. The field $F\langle \g \rangle$
generated by the eigenvalues of $\g$ over $F$ is well-defined.
An element in $ \G$ having three distinct
eigenvalues in $F$ is called {\it split}~; it is called {\it
rank-one split} if it has { distinct eigenvalues but only one of them lies in $F$}. In the
latter case we say it is {\it unramified}/{\it ramified} rank-one
split if its eigenvalues generate an unramified/ramified quadratic
extension of $F$. Note that $\G$ does not contain elements with no
eigenvalue in $F$. Indeed, if $\g$ is such an element, then the
characteristic polynomial of any lifting $\tilde \g$ of $\g$ in $\GL_3(F)$ is
irreducible over $F$. As $\ord_{\pi}
(\det \tilde \g) = 3m$ for some integer $m$, the eigenvalues of $\tilde \g' =
\pi^{-m}\tilde \g$ are units in a cubic extension of $F$, which implies
that $\g$ lies in the intersection of $\G$ with a conjugate of $K$,
and hence is the identity element. Together with the fact that
every element in a discrete cocompact lattice is semisimple (see \cite{Ra} Thm.1.12),
we arrive at
\begin{theorem}[Classification of elements in $\Gamma$] $ $ \\
Every element $\g$ of $\Gamma$ falls in one of the following types:\\
1) $\g$ is the identity; \\
2) $\g$ is split, that is, it has three distinct eigenvalues in $F^{\times}$;
 \\
3) $\g$ is ramified/unramified rank-one split, that is, {$\g$ has three distinct eigenvalues and} $F\langle \g \rangle$ is a ramfield/unramified quadratic extension of $F$; \\
4) $\g$ is irregular, that is, it has a repeated eigenvalue with multiplicity two.
\end{theorem}

The following conclusion on $\G$ shown in \cite{KLW} results from the closed form expression of the zeta function identity of $X_\Gamma$.

\begin{proposition} [\cite{KLW}, Corollary 4]
$\Gamma$ contains rank-one split elements.
\end{proposition}

\subsection{Rational form}\label{rationalform}
Let $\g$ be a non-identity element in $\G$. If $F\langle \g \rangle=F$, we may
assume that the eigenvalues are $1, a, b \in F^\times$ with
$\ord_{\pi}$$ b \ge \ord_{\pi}$$ a \ge 0$. Then $\g$ is conjugate to
$r_\g := \diag(1, a, b)$. If $\g$ is rank-one split, then its
characteristic polynomial has the form $(x - a)(x^2 - b'x - c')$
with $x^2 - b'x - c'$ irreducible over $F$. The splitting field of
$x^2 - b'x - c'$ is a quadratic extension $L= F(\lambda)$ of $F$. We
fix the choice of $\lambda$ so that it is a unit if $L$ is
unramified over $F$ and it is a uniformizing element if $L$ is
ramified over $F$. Let $x^2 - bx - c$ be the irreducible polynomial
of $\lambda$ over $F$ and let $\bar {\lambda}$ be the Galois
conjugate of $\lambda$. Then $\ord_{\pi}$$ c = 0$ or $1$ according as
$L$ is unramified or ramified over $F$ and $\ord_{\pi}$$ b \ge \frac
{1}{2}\ord_{\pi}$$ c$. There are elements $e, d \in F$ such that $e +
d \lambda$ and $e + d \bar {\lambda}$ are the roots of $x^2 - b'x -
c'$ in $L$. Consequently $\g$ is conjugate to
$r_\g :=\left( \begin{matrix} a &  &  \\ & e & dc \\ & d& e+db \end{matrix}\right)$.
 We shall assume that all eigenvalues
of $r_\g$ are minimally integral. In other words, $a, e, d$ are
in $\oo$ and at least one of them is a unit.
Call $r_\g$ the {\it rational form of} $\g$. Clearly it depends on the conjugacy class of $\g$.

We study centralizers of $\g \in \G$.
\begin{proposition}\label{centralizer} Let $\g \in \G$ be a non-identity element.
\begin{itemize}
\item[(1)] If $\g$ is rank-one split, then its centralizer $C_G(\g) \cong F\langle \g \rangle^{\times}$ is a non-split torus, and $C_\G(\g)$ is a free abelian group of rank one.

\item[(2)] If $\g$ is split, then its centralizer $C_G(\g) \cong (F^{\times})^2$ is a split torus, and $C_\G(\g)$ is a free abelian group of rank two.

\item[(3)] If $\g$ is irregular, then its centralizer $C_G(\g) \cong \GL_2(F)$ is not a torus, and $C_\G(\g)$ is isomorphic to a discrete co-compact torsion-free subgroup of $\GL_2(F)$.
\end{itemize}
\end{proposition}

\begin{proof} (1) Assume $\g \in \Gamma$ is rank-one split. There is an element $h \in G$ such that
 $h\m \g h=  r_\g  = \left(
\begin{matrix} a &  &  \\ & e & dc \\ & d& e+db
\end{matrix}\right).$ Up to scalars we may express
$$h\m C_G(\g) h= \left\{
\begin{pmatrix} 1 &  &  \\ & x & cy \\ & y& x+by \end{pmatrix} ~|~  x, y \in F, \text{~not ~both} ~0 \right\}.$$
The map $\phi$ sending $g = \begin{pmatrix} 1 &  &  \\ & x & cy \\ & y& x+by \end{pmatrix}$ to $ x + \lambda y$ yields an isomorphism from $h\m C_G(\g) h$ to $F(\gamma)^\times$ such that the norm of $\phi(g)$ is equal to $\det g$.
As a torsion-free discrete subgroup of $C_G(\g)$, $C_\G(\g)$ is a free abelian group.
If $C_\G(\g)$ has rank greater than one, then $\phi(h\m C_\G(\g) h)$ contains a nontrivial unit $u = x+ \lambda y$. Thus $x, y \in \oo$ and the norm of $u$ is in $\oo^\times$. This means that $g = \phi^{-1}(u)$ has integral entries and $\det g$ is a unit in $\oo$. In other words, $g$ is a non-identity element in $K$. Hence $h g h\m$ is a non-identity element in $\G \cap hKh\m$, and thus has finite order, contradicting the torsion-free assumption of $\G$.
Therefore $C_\G(\g)$ is a free abelian group of rank one.

(2) When $\g$ is split, we have $h\m \g h=  r_\g  = \diag(1, a, b)$ for some $h \in G$, and $h\m C_G(\g) h$ can be expressed as $\{\diag(1, x, y) ~|~ x, y \in F^\times \}$, which is isomorphic to $F^\times \times F^\times$ under the map $\diag(1, x, y) \mapsto (x, y)$. Since $\G$ intersects any compact subgroup of $G$ trivially,
$ C_\G(\g)$ can be identified as a subgroup of $C_G(\g)/(C_G(\g) \cap K) \simeq (F^\times /\oo^\times) \times (F^\times /\oo^\times) \simeq \Z \times \Z$, thus it has rank at most $2$.
If $C_\G(\g)$ has rank less than $2$, then $C_\G(\g) \backslash C_G(\g)K/K$ is infinite, which contradicts the finiteness of $X_\G$. Therefore $C_\G(\g)$ is a rank two abelian group.

(3) When $\g$ is irregular, then $h\m \g h=  r_\g  = \diag(1, a, a)$ for some $h \in G$, and $h\m C_G(\g) h$
is clearly isomorphic to $\GL_2(F)$. Under this isomorphism, $h\m C_\G(\g) h$ is mapped to a discrete co-compact torsion-free subgroup of $\GL_2(F)$.
\end{proof}

In what follows, we assume that $\G$ satisfies the two additional conditions below:
\begin{itemize}
\item[(I)] $ \ord_{\pi}$$ \det \G \subset 3 \mathbb Z$
so that $\G$ identifies vertices of the same type, and consequently
$X_\G$ is a finite connected $(q+1)$-regular $2$-dimensional
simplicial complex.

\item[(II)] $\G$ is {\it regular}, that is, $\G$ does not contain irregular elements.
Equivalently, the centralizer in $G$ of
any non-identity element in $\G$ is a torus.
\end{itemize}

\begin{remark}The condition (II) is imposed to ease our computations. As shown in \cite{KLW} using representation-theoretic approach,
this assumption is not needed.
\end{remark}

\subsection{The type and lengths of a homotopy class}

 The type, geometric length
and algebraic length of a homotopy class $\kappa_\g(gK)$ of $X_\G$
are those of $\kappa_\g(gK)$ in $\B$. In other words, If $g^{-1}\g g
\in T_{n,m}$, then $\kappa_\g(gK)$ has algebraic length
$l_A(\kappa_\gamma(gK))=n + 2m$, geometric length
$l_G(\kappa_\g(gK))= n + m$, and type $(n,m)$. By assumption,
$\kappa_\g(gK)$ has positive length if and only if $\g$ is not
identity.

\subsection{The type and lengths of
$[\g]$}

Let $\g \in [\G]$ be non-identity, and let $r_\g$ be its rational form as defined in \S \ref{rationalform}.
Fix a choice of $P_\g \in G$ such that $r_\g = (P_\g)^{-1} \g P_\g$.  As the centralizers of $\g$
 and $r_\g$ are related by
  $C_G(\g) = P_\g C_G(r_\g)P_\g^{-1}$, we have $C_\G(\g)P_\g = P_\g C_{P_\g^{-1}\G
P_\g}(r_\g)$, and $[\g]$ may be expressed in two ways:
\begin{eqnarray}\label{classofgamma2}
[\g] &=& \{ \kappa_\g(gK) ~|~  g \in C_\G(\g)\backslash G/K \} \notag \\
&=& \{ \kappa_\g(P_\g gK) ~|~  g \in C_{P_\g^{-1} \G P_\g}
(r_\g)\backslash G/K \} .
\end{eqnarray}
The second expression will facilitate our computations later on.

 Suppose $r_\g \in T_{n,m}$. We say that
$[\g]$ has type $(n, m)$, algebraic length $l_A([\g]) = n + 2m$ and
geometric length $l_G([\g]) = n + m$. As before, call $[\g]$ of type $1$ or $2$ according as $m=0$ or $n=0$. We shall prove
\begin{theorem}\label{charlAandlG} Let $\g \in [\G]$ and $\g \ne id$. Then
\begin{eqnarray*}
 l_A([\g]) = min_{\kappa_\g(g K) \in [\g]} ~l_A(\kappa_\g(g K)) \quad \text{and} \quad
l_G([\g]) = min_{\kappa_\g(g K) \in [\g]} ~l_G(\kappa_\g(g K)).
\end{eqnarray*}
Moreover, for $g \in C_G(r_\g)$, we have $l_A(\kappa_\g(P_\g gK))=
l_A([\g])$, $l_G(\kappa_\g(P_\g gK))= l_G([\g])$ and the type of
$\kappa_\g(P_\g gK)$ coincides with the type of $[\g]$.
\end{theorem}

 The second assertion is obvious since $(P_\g g)^{-1} \g P_\g g = g^{-1} r_\g g = r_\g$ for
 $g \in C_G(r_\g)$.
 The proof of the first assertion is contained in Theorem \ref{minlength} for $\g$ split
 and Theorem \ref{rankoneminlength} for $\g$
 rank-one split.

Note that $l_A(\kappa_\gamma(gK)) \equiv \ord_{\pi} \det \g  \pmod
3$, hence $l_A(\kappa_\g(gK)) = l_A([\g]) + 3m$ for some
non-negative integer $m$.

\subsection{Tailless cycles}\label{taillesscycles}

In view of Theorem \ref{charlAandlG},
 a homotopy class $\kappa_\g(gK)$ is called {\it
algebraically tailless} if its algebraic length agrees with
$l_A([\g])$. It is called {\it tailless} if its geometric length is
$l_G([\g])$.
By Proposition \ref{basepointfreehomotopy},
 a tailless {vertex-}based homotopy cycle has shortest
geometric length among all cycles base-point free homotopic to it.

\subsection{The volume of $[\g]$}\label{volume}
By Proposition \ref{centralizer} and assumption (II), $C_G(r_\g)$ is a torus in
$G$ containing the discrete cocompact subgroup $C_{P_\g^{-1}\G
P_\g}(r_\g)$.  Let
$$\Omega = \cup_{g \in G} ~g^{-1}Kg.$$
Then the double coset $C_{P_\g^{-1}\G
P_\g}(r_\g)\backslash C_G(r_\g)/(C_G(r_\g)\cap \Omega)$ is finite.
Its cardinality is the same as that of $C_\G(\g) \backslash C_G(\g)/ C_G(\g) \cap \Omega$, called the {\it volume of} $[\g]$:
\begin{eqnarray}\label{fundamentaldomain}
\vol([\g]) = \# \big(C_{P_\g^{-1} \G P_\g}(r_\g) \backslash
C_G(r_\g)/(C_G(r_\g)\cap \Omega)\big) = \# \big(C_\G(\g) \backslash C_G(\g)/ C_G(\g) \cap \Omega \big).
\end{eqnarray}

Observe that
\begin{lemma}\label{Omega} For any $g \in G$, we have $C_G(r_\g) \cap gKg^{-1} = C_G(r_\g) \cap gKg^{-1} \cap K$. Consequently,
$C_G(r_\g)\cap \Omega = C_G(r_\g)\cap K$.
\end{lemma}

\begin{proof} Let $g \in G$. It suffices to show
$C_G(r_\g)\cap gKg^{-1} \subset K$. Suppose
$h \in C_G(r_\g) \cap gKg^{-1}$.
 Without loss of generality, we may assume that the
eigenvalues of $h$ are roots of $f(x)$, the characteristic polynomial of
some element in $GL_3(\oo)$. We distinguish two cases.

Case I. $\g$ is split.  Then $C_G(r_\g)$ consists of diagonal matrices in $G$. Let $\alpha, \mu, \nu$
be the diagonal entries of $h$ with
$\ord_\pi \alpha \ge \ord_\pi \mu \ge \ord_\pi \nu$. Then
$f(x)= (x - \alpha)(x - \mu)(x - \nu)$ lies in $\oo[x]$. The constant term of $f(x)$ is a unit
in $\oo$, which implies $\ord_\pi \alpha \ge 0 \ge \ord_\pi \nu$. Thus if $\alpha, \mu, \nu$
have the same order, then they are all units. If not, then $\ord_\pi \alpha > 0$ and $\ord_\pi \nu < 0$.
Since the coefficient of $x$ in $f(x)$ is in $\oo$, we have $\ord_\pi \mu \nu \ge 0$, which contradicts
$\ord_\pi \alpha \mu \nu = 0$. Therefore $h \in K$.

Case II. $\g$ is rank-one split. Then, as in  the proof of Proposition \ref{centralizer},(1),
 $h = \left(
\begin{matrix} \alpha &  &  \\ & \mu & c \nu \\ & \nu& \mu+ b \nu
\end{matrix}\right)$ with eigenvalues $\alpha, \mu + \nu \lambda, \mu + \nu \bar{\lambda}$.
Here $\lambda$ has minimal polynomial $x^2 - bx - c$ over $F$, and $\lambda$ is either a unit of a uniformizer
in the field $F(\lambda)$.
Thus $\beta:=\alpha (\mu^2 + \mu \nu b - \nu^2 c)$ is a unit
in $F$, and $\alpha + 2 \mu + b \nu$ and $\delta:=\alpha (2 \mu + b \nu) + \mu^2 + \mu \nu b - \nu^2 c$
both lie in $\oo$. If $\ord_\pi \alpha < 0$, then $\ord_\pi $$(2 \mu + b \nu) = \ord_\pi $$\alpha < 0$ so that
$\ord_\pi $$\alpha (2 \mu + b \nu) < 0$ while
$\ord_\pi $$(\mu^2 + \mu \nu b - \nu^2 c) = \ord_\pi \beta/\alpha = - \ord_\pi \alpha > 0$, contradicting
$\delta \in \oo$. If $\ord_\pi \alpha > 0$, then
$2 \mu + b \nu \in \oo$ so that $\alpha (2 \mu + b \nu) \in \oo$ while
$\ord_\pi $$(\mu^2 + \mu \nu b - \nu^2 c) < 0$. We obtain the same contradiction. Thus
$\alpha$ is a unit, and so are $\mu + \nu \lambda$ and $\mu + \nu \bar{\lambda}$. The choice of $\lambda$
implies $\mu, \nu \in \oo$. Hence
$h \in K$, as desired.
\end{proof}

Thus we can express $\vol([\g])$ as
\begin{eqnarray}\label{fundamentaldomain'}
\vol([\g]) = \# \big(C_{P_\g^{-1} \G P_\g}(r_\g) \backslash
C_G(r_\g)/(C_G(r_\g)\cap K)\big).
\end{eqnarray}

\begin{remark}For any integer $m \ne 0$, the eigenvalues of $\g^m$ are the $m$-th
power of those of $\g$, hence $r_{\g^m} =(r_\g)^m$ up to a central
element (due to normalization), and thus we may assume $P_{\g^m} =
P_\g$. Consequently, $P_{\g^m}^{-1} \G P_{\g^m} = P_\g^{-1} \G P_\g$
for all $m \ne 0$. Clearly, $C_G(r_\g) \subseteq C_G(r_{\g^m})$. The reverse containment
follows from the argument in the proof of Proposition \ref{centralizer}. Therefore
$C_G(r_{\g^m}) = C_G(r_\g)$. This shows
 $\vol([\g]) = \vol([\g^m])$ for all $m \ne 0$.
 \end{remark}

\section{Hecke operators on $\B$ and on $X_\G$}

\subsection{Recursive relations among Hecke operators}
It is well-known that each Hecke operator is a polynomial in $A_1$ and
$A_2$. Tamagawa \cite{Ta} obtained a recursive relation on Hecke operators:

\begin{eqnarray}\label{rationalTnm}
(\sum_{n,m \geq 0} T_{n,m}u^{n+2m})(I-A_1 u + q A_2 u^2 - q^3 u^3 I)
= (1 - u^3)I.
\end{eqnarray}

\noindent We prove a different recursive formula adapted for our needs.

\begin{theorem}
\begin{eqnarray}\label{recursivehecke}
~\qquad q \sum^{\infty}_{k=1} T_{k,0} u^k -
(q-1)(\sum^{\infty}_{k=1} \sum_{n+2m = k} T_{n,m}
u^k)\frac{1-q^2u^3}{1-u^3}= u \frac{d}{du}
\log\frac{(1-u^3)^{r}I}{I-A_1 u + A_2 q u^2 -q^3u^3 I},
\end{eqnarray} where $ r=\frac{(q+1)(q-1)^2}{3}$.
\end{theorem}

\begin{proof}
The algebra of Hecke operators is isomorphic to
the polynomial ring $\mathbb{C}[z_1,z_2,z_3]^{S_3} /
\langle z_1z_2z_3 - 1 \rangle$ under the Satake isomorphism $\psi$ (cf. \cite{Sat}).
To describe its values on $\{T_{n,m}\}$, let $\chi$ be the
quasi-character on the Borel subgroup $P$ of $G$ defined by
$$ \chi \left( \left[\begin{matrix} b_1 & * & * \\ & b_2 & * \\ & & b_3 \end{matrix}\right]\right)
= z_1^{\ord_\pi (b_1)}z_2^{\ord_\pi (b_2)}z_3^{\ord_\pi (b_3)},  $$
and regard it as a map from $G/K$ to $\mathbb{C}[z_1,z_2,z_3] /
\langle z_1z_2z_3 - 1 \rangle$. (The relation $z_1z_2z_3 = 1$ follows from the fact
that $\chi$ is trivial on scalar matrices.)
Denote by $\delta_P$ the modular
character on $P$. Let $\phi$ be the function on $G$ given by
$$\phi(bk)=\chi(b)\delta_P^{1/2}(b) \qquad (b\in P, k \in K).$$
\noindent Then the value of the Satake isomorphism at $T_{n,m}$ is
$$\psi(T_{n,m}) = \int_G T_{n,m}(g)\phi(g) dg \qquad (n \ge 0, ~m \ge 0),$$ where $dg$ is the
Haar measure on $G$ so that $K$ has volume $1$. Direct computations give
$\psi(A_1)=q(z_1 + z_2 + z_3)$, $\psi(A_2)=q(z_1z_2 + z_2z_3 + z_3z_1)$ and
$$ \psi(I-A_1 u + q A_2 u^2 - q^3 u^3 I) = (1-qz_1u)(1-qz_2u)(1-qz_3u).$$

 For $k \ge 1$, let $T_k = \sum_{n+2m=k} T_{n,m}$, and set
$$\sigma_{k,1}(z_1,z_2,z_3)= z_1^k + z_2^k + z_3^k,  \quad \sigma_{k,2}(z_1,z_2,z_3)= \sum_{1\le a \le k-1} z_1^a z_2^{k-a} + z_2^a z_3^{k-a} + z_3^a z_1^{k-a},$$ and $$\sigma_{k,3}(z_1,z_2,z_3)= \sum_{a, b, c \ge 1, a+b+c=k}z_1^a z_2^b z_3^c. $$
Our strategy is to show that the identity (\ref{recursivehecke}) holds after applying the
Satake isomorphism $\psi$. For this, it suffices to compute the coefficient of
$z_1^{a_1}z_2^{a_2}z_3^{a_3}$ in $\psi(T_k)$ with $a_1 \ge a_2 \ge
a_3 \ge 0$ and $a_1+ a_2 + a_3 = k$, then use symmetry to determine $\psi(T_k)$.

It is straightforward to check that the number of elements
 $gK \in \bigsqcup_{n+2m=k}  T_{n,m} $ mapped to $z_1^{a_1}z_2^{a_2}z_3^{a_3}$ by $\chi$ is equal to
$q^{2a_1+a_2}$ if $a_3=0$, and $(q^3-1) q^{2a_1+a_2-3}$ if $a_3>0$.
Moreover, for such $gK$ we have $\delta_P(gK)^{1/2}= q^{a_3-a_1}$.
Therefore the coefficient of $z_1^{a_1}z_2^{a_2}z_3^{a_3}$ in
$\psi(T_k)$ is equal to $q^{a_1+a_2+a_3}$ or
$q^{a_1+a_2+a_3-3}(q^3-1)$ according to $a_3 = 0$ or $a_3 > 0$. By
symmetry, this yields
$$\psi(T_k)= q^{k}(\sigma_{k,1}+\sigma_{k,2}+\frac{q^3-1}{q^3}\sigma_{k,3}).$$
Noting that
$$\sum^{\infty}_{k=1}\sigma_{k,3}u^k=
((z_1z_2z_3)u^3+(z_1z_2z_3)^2u^6+\cdots)\sum^{\infty}_{k=0}(1+\sigma_{k,1}+\sigma_{k,2})u^k
=
\frac{u^3}{1-u^3}\sum^{\infty}_{k=0}(1+\sigma_{k,1}+\sigma_{k,2})u^k,$$
we obtain
\begin{eqnarray*}
\psi( \sum^{\infty}_{k=1} T_k u^k )
&=& \sum^{\infty}_{k=1} (\sigma_{k,1}+\sigma_{k,2}+\frac{q^3-1}{q^3}\sigma_{k,3})(qu)^k \\
&=& \frac{(q^3-1)u^3}{1-q^3u^3}+\frac{1-u^3}{1-q^3u^3}\sum^{\infty}_{k=1} (\sigma_{k,1}+\sigma_{k,2})(qu)^k.
\end{eqnarray*}

On the other hand, put $G_0 = \bigsqcup^{\infty}_{k=1} T_{k,0}$. One
verifies that the number of elements in $G_0/K$ mapped to
$z_1^{a_1}z_2^{a_2}z_3^{a_3}$ by $\chi$ is $q^{2a_1}$ if
$a_2=a_3=0$, $(q-1) q^{2a_1+a_2-1}$ if $a_2>a_3=0$, and
$(q-1)^2q^{2a_1+a_2-2}$ if $a_2\geq a_3>0$. Therefore,
\begin{eqnarray*}
\psi( \sum^{\infty}_{k=1} T_{k,0} u^k )&=& \sum^{\infty}_{k=1} (\sigma_{k,1}+\frac{q-1}{q}\sigma_{k,2}+\frac{(q-1)^2}{q^2}\sigma_{k,3})(qu)^k \\
&=& \frac{q(q-1)^2u^3}{1-q^3u^3}+\frac{1+q u^3 - 2 q^2
u^3}{1-q^3u^3}\sum^{\infty}_{k=1}\sigma_{k,1}(qu)^k
+\frac{(q-1)(1-q^2u^3)}{q(1-q^3u^3)}\sum^{\infty}_{k=1}\sigma_{k,2}(qu)^k.\\
\end{eqnarray*}
Consequently,
\begin{eqnarray*}
& & \psi\bigg(q (\sum^{\infty}_{k=1} T_{k,0} u^k) -
(q-1)(\sum^{\infty}_{k=1} T_k u^k)\frac{1-q^2u^3}{1-u^3}\bigg)
=\sum^{\infty}_{k=0}\sigma_{k,1}(qu)^k + \frac{(q-1)(q^2-1)u^3}{1-u^3}\\
 &=& \frac{z_1 qu}{1-z_1 qu}+\frac{z_2 qu}{1-z_2 qu}+\frac{z_2 qu}{1-z_2 qu} - \frac{3ru^3}{1-u^3}
= u \frac{d}{du} \log \frac{(1-u^3)^r}{(1-z_1 qu)(1-z_2 qu)(1-z_3 qu)}\\
&=& \psi\bigg(u \frac{d}{du} \log\frac{(1-u^3)^r}{I-A_1 u + A_2 q
u^2 -q^3u^3 I} \bigg).
\end{eqnarray*}
\end{proof}

\subsection{Hecke operators on $X_\G$}

The action of the Hecke operator $T_{n,m}$ on $L^2(\G \backslash G/K)$ is represented by the matrix $B_{n,m}$, whose rows and columns are indexed by vertices of $X_\G$
such that the $(\G gK, \G g'K)$ entry records the number of homotopy
classes of geodesic paths from $\G gK$ to $\G g'K$ in $X_\G$ of type
$(n, m)$. Alternatively, this is the number of $\g \in \G$ such that
the homotopy classes of the geodesics from $gK$ to $\g g'K$ have
type $(n, m)$. The trace of $B_{n,m}$ then gives the number of
geodesic cycles of type $(n,m)$ up to homotopy. In other words,
\begin{eqnarray*}
\Tr(B_{n,m}) &=& \#\bigg\{~\kappa_\g(gK) ~|~ \g \in [\G],
~\kappa_\g(gK) \in [\g] \text{ has type }(n,m) \bigg\}.
\end{eqnarray*}
To facilitate our computations, form two kinds of formal power
series:
\begin{equation}\label{sumBnm}
 \sum_{\substack{n,m \geq 0 \\ (n,m)\neq (0,0)}}
\Tr(B_{n,m})u^{n+2m} = \sum_{\g \in [\G], ~\g \ne id}
~\sum_{\kappa_\g(gK) \in [\g]} ~u^{l_A(\kappa_\g(gK))},
\end{equation} and
\begin{equation}\label{sumBn0}
 \sum_{n>0} \Tr(B_{n,0})u^{n}
= \sum_{\g \in [\Gamma], ~\g \ne id} ~\sum_{ \kappa_\g(gK) \in [\g]
\text{ has type $1$}}  ~u^{l_A(\kappa_\g(gK))}.
\end{equation}

Now we rewrite the left hand side of the zeta identity (\ref{zetaidentity}) as
 \begin{proposition}
\begin{eqnarray}\label{lefthand1}
&& u \frac{d}{du}
\log\frac{(1-u^3)^{\chi(X_\G)}}{\det(I-A_1 u + A_2 q u^2 -q^3u^3 I)}\\
&=&q\left(\sum_{n>0} \Tr(B_{n,0}) u^n\right)-(q-1)\left( \sum_{\substack{n,m \geq 0 \\ (n,m)\neq (0,0)}}
\Tr(B_{n,m})u^{n+2m}\right)\frac{1-q^2u^3}{1-u^3}, \notag
\end{eqnarray}
where the operators are on $L^2(\G \backslash G/K)$, $\chi(X_\G)= \frac{(q+1)(q-1)^2}{3}V$ is the Euler characteristic of $X_\G$, and $V$ is the number of vertices
in $X_\G$.
\end{proposition}
\begin{proof}
Note that $B_{n,m}$ is $T_{n,m}$ acting on the space $L^2(\G \backslash G/K)$, so
(\ref{recursivehecke}) also holds with $T_{n,m}$ replaced by $B_{n,m}$. In other words,\begin{eqnarray*}
& & u \frac{d}{du}
\Tr \log\frac{(1-u^3)^rI}{(I-A_1 u + A_2 q u^2 -q^3u^3 I)} \\
&=& q\left(\sum_{n>0} \Tr(B_{n,0}) u^n\right)-(q-1)\left( \sum_{\substack{n,m \geq 0 \\ (n,m)\neq (0,0)}}
\Tr(B_{n,m})u^{n+2m}\right)\frac{1-q^2u^3}{1-u^3},
\end{eqnarray*} where $r = \frac{(q+1)(q-1)^2}{3}$.
Recall that each vertex is incident to $q^2+q+1$ type 1 edges and $q^2+q+1$ type 2
edges so that the total number of undirected edges in $X_\G$ is
$\frac{2(q^2+q+1)}{2}V$. Since each edge is contained in $q+1$ chambers,
the number of chambers in $X_\G$ is
$\frac{(q+1)}{3}(q^2+q+1)V$.  Therefore the Euler characteristic of
$X_\G$ is
$$ \chi(X_\G) = V - (q^2+q+1) V + \frac{(q+1)}{3}(q^2+q+1)V =
\frac{(q-1)^2(q+1)}{3}V = rV.$$
Using the identity
$$ \log( \det A ) =  \Tr ( \log  A)$$
for a $V\times V$ matrix $A$, we have
$$u \frac{d}{du}
\Tr \log\frac{(1-u^3)^rI}{(I-A_1 u + A_2 q u^2 -q^3u^3 I)}=u \frac{d}{du}
 \log\frac{(1-u^3)^{\chi(X_\G)}}{\det(I-A_1 u + A_2 q u^2 -q^3u^3 I)}, $$
which proves the proposition.
\end{proof}
To understand the combinatorial meaning of the right hand side of (\ref{lefthand1}), we first determine the algebraic length of $\kappa_\g (gK)$, then compute
$\sum_{\kappa_\g (gK)\in [\gamma]} u^{l_A(\kappa_\g (gK))}$ and
$\sum_{\kappa_\g (gK)\in [\gamma] ~\text{has type $1$} } u^{l_A(\kappa_\g (gK))}$.

\section{Homotopy cycles in $[\g]$ for $\g$ split}
Let $|~|$ be the valuation on $F$ such that $|\pi| = q^{-1}$.
In this section we fix a split $\g \in [\G]$ with rational form
$r_\g = \diag(1, a, b)$, where $\ord_{\pi}$$ b \ge \ord_{\pi}$$ a \ge
0$.

\subsection{Minimal lengths of homotopy cycles in $[\g]$}

We begin by proving the first assertion of Theorem \ref{charlAandlG}
for the split case.

\begin{theorem}\label{minlength}
Suppose $\g \in \G$ is split with $r_\g = \diag(1, a, b)$, where
$\ord_{\pi}$$ b \ge \ord_{\pi}$$ a \ge 0$. Then
\begin{itemize}
\item[(1)]
$l_A([\g])= \ord_{\pi}$$ a + \ord_{\pi}$$ b = min_{\kappa_\g(gK) \in
[\g]} ~l_A(\kappa_\g(gK))$ and

\item[(2)] $l_G([\g])=
 \ord_{\pi}$$ b = min_{\kappa_\g(gK) \in
[\g]} ~l_G(\kappa_\g(gK))$.
\end{itemize}
\end{theorem}
\begin{proof} The centralizer
$C_G(r_\g)$ consists of the diagonal matrices in $G$ so that $G = C_G(r_\g)UK$, where
$$ U = \bigg \{\left(\begin{matrix} 1 & x & y \\ & 1 & z \\ & &
1\end{matrix}\right) ~|~ x, y, z \in F/\oo \bigg \}.$$ It suffices to consider the
lengths of $\kappa_\g(P_\g gK)$ with $g \in U$. Write $ g =
\left(\begin{matrix} 1 & x & y \\ & 1 & z \\ & &
1\end{matrix}\right)$.  Then
\begin{eqnarray*}(P_\g g)^{-1} \g P_\g g = g\m r_\g g &=& \left(\begin{matrix} 1 & x & y \\ & 1 & z \\ & &
1\end{matrix}\right)\m \left(\begin{matrix} 1 &  &  \\ & a &  \\
& & b\end{matrix}\right)\left(\begin{matrix} 1 & x & y \\ & 1 & z \\
& & 1\end{matrix}\right) \\
&=&   \left(\begin{matrix} 1 & x(1-a) & y(1-b)+xz(b-a) \\
& a & z(a-b) \\ & & b\end{matrix}\right) \in
K \left(\begin{matrix} \pi^{e_1} &  &  \\
& \pi^{e_2} & \\ & & \pi^{e_3}\end{matrix}\right) K
\end{eqnarray*}
for some integers $e_1 \le e_2 \le e_3$. In fact, for $1 \le i \le
3$, $e_1 + \cdots + e_i = min_{y} ~\{\ord_{\pi}$$ y\}$ where $y$ runs
through the determinant of all $i \times i$ minors of $g^{-1}r_\g
g$.
Consequently,
\begin{eqnarray}\label{e1}
e_1 = min \{0, ~\ord_{\pi} x(1-a), ~\ord_{\pi} z(a-b), ~\ord_{\pi}
(y(1-b)+xz(b-a)) \} \leq 0,
\end{eqnarray}
\begin{eqnarray}\label{e1+e2}
e_1+e_2 = min \{\ord_{\pi} a, ~\ord_{\pi} [x(1-a)z(a-b) -
a(y(1-b)+xz(b-a))] \} \leq \ord_{\pi} a, \end{eqnarray} and
\begin{eqnarray}\label{e1+e2+e3}
e_1+e_2+e_3 = \ord_{\pi} a + \ord_{\pi} b.
\end{eqnarray}
In particular, $e_3 \ge \ord_{\pi}$$ b$ from the last two
inequalities. Moreover, we have, for any $g \in G$,
\begin{eqnarray}\label{alglengthbound}
~~l_A(\kappa_\g(P_\g gK))=e_3+e_2+e_1-3e_1=\ord_{\pi} a + \ord_{\pi}
b - 3e_1\geq \ord_{\pi} a + \ord_{\pi} b = l_A([\g])
\end{eqnarray}
 and
\begin{eqnarray}\label{geomlengthbound}
 l_G(\kappa_\g(P_\g gK))=e_3-e_1\geq
\ord_{\pi} b - e_1 \geq \ord_{\pi} b = l_G([\g]).
\end{eqnarray}
As noted before, the equalities in (\ref{alglengthbound}) and
(\ref{geomlengthbound}) hold for $g \in C_G(r_\g)$. Therefore
\begin{eqnarray*}
l_A([\g]) = \min_{\kappa_\g(gK) \in [\g]} l_A(\kappa_\g(gK)) \qquad
\text{ and} \qquad l_G([\g]) = \min_{\kappa_\g(gK) \in [\g]}
l_G(\kappa_\g(gK)).
\end{eqnarray*}
This proves the theorem.
\end{proof}

It follows from (\ref{geomlengthbound}) that if $l_G(\kappa_\g(P_\g
gK)) = l_G([\g]) = \ord_{\pi}$$ b$, then $e_3 = \ord_{\pi}$$ b$ and $e_1
= 0$, which in turn imply $e_2 = \ord_{\pi}$$ a $ because $e_1 + e_2 +
e_3 = \ord_{\pi}$$ a + \ord_{\pi}$$ b$. Hence a  tailless cycle
$\kappa_\g(P_\g gK)$ in $[\g]$  has the same type as $[\g]$.
Further, by (\ref{alglengthbound}), the condition $e_1 = 0$ implies
$l_A(\kappa_\g(P_\g gK)) = l_A([\g])$, so $\kappa_\g(P_\g gK)$ is
also algebraically tailless.

Conversely, suppose $\kappa_\g(P_\g gK)$ is algebraically tailless.
Then $e_1 = 0$, that is, $x(1-a) \in \oo$, $z(a-b) \in \oo$ and
$y(1-b)+xz(b-a) \in \oo$. As seen above, $\kappa_\g(P_\g gK)$ is
 tailless if the additional condition  $e_1 + e_2 =
\ord_{\pi}$$ a$ is satisfied. By (\ref{e1+e2}), this amounts to
$\ord_{\pi}$$ x(1-a)z(a-b) \ge \ord_{\pi}$$ a$, which obviously holds
when $\ord_{\pi}$$ a = 0$, i.e., $\g$ has type $1$. We record the
discussion in

\begin{corollary}\label{typeoftailless}
 Suppose $\g \in [\G]$ is split. Then all tailless
cycles in $[\g]$ are also algebraically tailless, and they have the
same type as $[\g]$. Furthermore, if $[\g]$ has type $1$, then the
algebraically tailless cycles in $[\g]$ are
 tailless.
\end{corollary}

\subsection{Counting homotopy cycles in $[\g]$ in algebraic length}

Let
$$\Delta_A([\g])= \{ gK \in G/K ~|~
l_A(\kappa_\g(P_\g gK)) = l_A([\g])\}.$$ As noted before,
$\Delta_A([\g]) \supset C_G(r_\g)K/K$ and is invariant under left
multiplication by $C_{P_\g^{-1}\G P_\g}(r_\g)$. So the number of
algebraically tailless cycles in $[\g]$ is the cardinality of
$C_{P_\g^{-1}\Gamma P_\g}(r_\g)\backslash \Delta_A([\g])$.

The following theorem, stated in terms of a formal power series,
gives the number of homotopy cycles of a given algebraic length in
$[\g]$.

\begin{theorem}\label{numberinaclass}
Suppose $\g \in [\G]$ is split with $r_\g = \diag(1, a, b)$. Then
\begin{eqnarray*}
\sum_{\kappa_\g(gK) \in [\g]} u^{l_A(\kappa_\g(gK))}&=&
\#(C_{P_\g^{-1} \G P_\g}(r_\g)\backslash
\Delta_A([\g]))~u^{l_A([\g])}\frac{1-u^3}{1-q^3 u^3} \\
&=& \vol([\g])(|1-a||a-b||b-1|)^{-1}~u^{l_A([\g])}\frac{1-u^3}{1-q^3
u^3},
\end{eqnarray*} where $\vol([\g])$ is given by
(\ref{fundamentaldomain'}).
\end{theorem}
\begin{proof} The group $C_G(r_\g) \cap K$ consists of diagonal matrices
whose nonzero entries are units.
In view of Proposition \ref{centralizer}, there are two generators $s, t \in C_G(r_\g)$
such that $C_G(r_\g) = \langle s, t \rangle (C_G(r_\g) \cap K)$ and $C_{P_\g^{-1}\G P_\g}(r_\g)$
is a subgroup of $\langle s, t \rangle$ of index $\vol([\g])$. We have
$(C_G(r_\g) \cap K)UK = UK$ and $C_{P_\g^{-1}\G P_\g}(r_\g)\backslash G/K
= C_{P_\g^{-1}\G P_\g}(r_\g)\backslash \langle s, t \rangle UK/K$. Suppose
$h, h' \in \langle s, t \rangle$ and $v, v' \in U$ are such that
$C_{P_\g^{-1}\G P_\g}(r_\g)h vK = C_{P_\g^{-1}\G P_\g}(r_\g)h' v'K$.
Replacing $h$ by a suitable multiple from $C_{P_\g^{-1}\G P_\g}(r_\g)$ if necessary,
we may assume $hvK = h'v'K$, which is equivalent to $v^{-1}h^{-1}h'v' \in K$.
Since $v$ and $v'$ are unipotent and $h^{-1}h'$ is diagonal, $v^{-1}h^{-1}h'v'$ is
an upper triangular matrix with diagonal entries being those of $h^{-1}h'$. This
implies that $h^{-1}h' \in K$ and hence is equal to the identity matrix. It then follows
from the definition of $U$ that $v = v'$. This proves that
the left hand side of the identity
can be expressed as
$$\sum_{\kappa_\g(P_\g gK) \in [\g]} u^{l_A(\kappa_\g(P_\g gK))} =
\vol([\g]) ~\sum_{v \in U}  u^{l_A(\kappa_{\g}(P_\g vK))}.$$
To proceed, we compute the sum on the right hand side.

\begin{proposition}\label{sumoverunipotent}
 Let $\g$ be split with $r_\g = \diag(1,a,b)$. Then
\begin{eqnarray*}
\sum_{v \in U} u^{l_A(\kappa_{\g}(P_\g vK))} &=&
\frac{u^{l_A([\g])}}{|1-a||a-b||b-1|}\bigg(\frac{1-
u^3}{1-q^3u^3}\bigg).
\end{eqnarray*}
\end{proposition}
\begin{proof}
Given $v \in U$, write $v = \left(
\begin{matrix}
1 & x & y\\
& 1 & z\\
& & 1
\end{matrix} \right)$.
As computed in the proof of Theorem
\ref{minlength},

$$ (P_\g v)^{-1} \g P_\g v =  v^{\text{-}1}r_\g v =
\left(
\begin{matrix}
1 & x(1-a) & y(1-b)+ xz(b-a)\\
& a & z(a-b)\\
& & b
\end{matrix} \right) = (v_{i, j}).$$
For fixed $m \ge 0$, we count the number of $v$'s such that
$l_A(\kappa_{\g}(P_\g vK)) \leq l_A([\g])+3m$. By
(\ref{alglengthbound}), the constraints are  $ |v_{ij}| \leq q^{m}$
for all $1 \le i,j \le 3$. In other words,
$$ |x(1-a)| \leq q^{m},\quad |z(a-b)|\leq q^{m}
\quad {\rm and} \quad |y(1-b)+ xz(b-a)|\leq q^{m}.$$  This implies
$$ |x| \leq q^{m}|1-a|^{-1} \quad {\rm and} \quad |z| \leq
q^{m}|a-b|^{-1} $$ so that the numbers of $x$ and $z$ in
$F/\mathcal{O}_F$ are $q^{m}|1-a|^{-1}$ and $q^m|a-b|^{-1}$,
respectively. Further, for chosen $x$ and $z$, there are
$q^{m}|1-b|^{-1}$ choices of $y$ satisfying the above constraint. We
have shown
\begin{eqnarray}\label{taillessinU}
\#\big\{v \in U \big| ~l_A(\kappa_{\g}(P_\g vK))=l_A([\g]) \big\} =
(|1-a||a-b||b-1|)^{-1}
\end{eqnarray} and, for $m > 0$,
\begin{eqnarray}\label{taillength3minU}
\#\big\{v \in U \big| ~l_A(\kappa_{\g}(P_\g vK))=l_A([\g])+3m \big\}
= (q^{3m}-q^{3m-3})(|1-a||a-b||b-1|)^{-1}.
\end{eqnarray}
Put together, this gives
\begin{eqnarray*} \sum_{v \in U} u^{l_A(\kappa_{\g}(P_\g vK))} &=&
\frac{u^{l_A([\g])}}{|1-a||a-b||b-1|}\bigg(1+ \sum_{m \ge 1}(q^{3m}-q^{3m-3})u^{3m}\bigg)\\
&=&\frac{u^{l_A([\g])}}{|1-a||a-b||b-1|}\bigg(\frac{1-
u^3}{1-q^3u^3}\bigg).
\end{eqnarray*}
\end{proof}

The argument above shows that the number of algebraically tailless
homotopy classes in  $[\g]$ is $\vol([\g])$ times the number of
elements in $U$ with $m = 0$, which is given by (\ref{taillessinU}).
This proves

\begin{proposition}\label{atailless}
Suppose $\g \in \G$ is split with $r_\g = \diag(1, a, b)$. Then
$$ \#(C_{P_\g^{-1} \G P_\g}(r_\g)\backslash
\Delta_A([\g])) = \vol([\g])(|1-a||a-b||b-1|)^{-1}.$$
\end{proposition}

The proof of Theorem \ref{numberinaclass} is now complete.
\end{proof}

\subsection{Counting homotopy cycles of type $1$ in $[\g]$}

The theorem below gives the number of type $1$ homotopy cycles in
 $[\g]$ of given algebraic length. The result depends on
the type of $[\g]$.

\begin{theorem}\label{type0cyclesinaclass} Suppose $\g \in \G$ is split with
$r_\g = \diag(1, a, b)$.
The following assertions hold.
\begin{itemize}
\item[(i)] If $[\g]$ does not have type $1$, then
$$\sum_{\kappa_\g(gK) \in [\g], ~ \text{type $1$}} u^{l_A(\kappa_\g(gK))}=
\vol([\g])(|1-a||a-b||b-1|)^{-1}~u^{l_A([\g])}(1 - q^{-1})( \frac{1-q^2
u^3}{1-q^3u^3}).$$ Moreover, no type $1$ cycles in $[\g]$ are
tailless.

 \item[(ii)] If $[\g]$ has type $1$, then
$$\sum_{\kappa_\g(gK) \in [\g], ~ \text{type $1$}} u^{l_A(\kappa_\g(gK))}=
\vol([\g])(|1-a||a-b||b-1|)^{-1}~u^{l_A([\g])}\bigg(q^{-1}+ (1 -
q^{-1})( \frac{1-q^2 u^3}{1-q^3u^3}) \bigg).$$
\end{itemize}
\end{theorem}

\begin{remark} The right hand side of the identities in Theorem \ref{numberinaclass}
and Theorem \ref{type0cyclesinaclass} can be expressed as $\vol([\g])$ times the orbital integrals at the split
element $\g$ of suitably chosen spherical functions on $G$ with fast decay.
\end{remark}

\begin{proof} Since $r_\g = \diag(1, a, b)$, $[\g]$ has type
$(\ord_{\pi}$$ b - \ord_{\pi}$$ a, ~\ord_{\pi}$$ a)$ and $l_A([\g]) =
\ord_{\pi}$$ b + \ord_{\pi}$$ a$. It has type $1$ if and only of
$~\ord_{\pi}$$ a = 0$. The argument is similar to the proof of Theorem
\ref{numberinaclass}; the difference is that we only need to
consider those $v \in U$ such that $\kappa_{\g}(P_\g vK)$ has type
$1$. So we count the number of
$$\{v \in U~ |~l_G(\kappa_{\g}(P_\g vK)) =l_A(\kappa_{\g}(P_\g vK)) =
l_A([\g])+3m = \ord_{\pi} ~b + \ord_{\pi} ~a + 3m \}$$ for each $m
\ge 0$. As before, writing $v$ as $\left(
\begin{matrix}
1 & x & y\\
& 1 & z\\
& & 1
\end{matrix} \right)$ and following the proofs of Proposition
\ref{sumoverunipotent} and Theorem \ref{minlength}, we arrive at the
following constraints on $x, y, z \in F/\oo$:
\begin{itemize}
\item[(1)] $\min \{0, ~\ord_{\pi}$$ x(1-a), ~\ord_{\pi}$$ z(a-b),
~\ord_{\pi}$$ (y(1-b)+xz(b-a)) \} = -m,$ ~~~~ and

\item[(2)]  $\min \{\ord_{\pi}$$ a, ~\ord_{\pi}$$ [x(1-a)z(a-b) -
a(y(1-b)+xz(b-a))] \} = -2m.$
\end{itemize}

For $m > 0$, the two constraints are equivalent to
\begin{itemize}
\item[(3)]  $ \ord_{\pi}$$ x(1-a) = -m = \ord_{\pi}$$ z(a-b) ~~\text{
and}~~ \ord_{\pi}$$(y(1-b)+xz(b-a)) \ge -m.$
\end{itemize}
Hence the number of $x$ is $(1-q^{-1})q^m|1-a|^{-1}$, the number of
$z$ is $(1-q^{-1})q^m|a-b|^{-1}$, and the number of $y$ is
$q^m|1-b|^{-1}$ so that the total number of $v$ is
$(1-q^{-1})^2q^{3m}(|1-a||a-b||b-1|)^{-1}$. For $m = 0$ and
$\ord_{\pi}$$ a > 0$, the same constraint (3) holds. In this case the
number of $x$ is $|1-a|^{-1} = 1$, the number of $y$ is $|1-b|^{-1}
= 1$ and the number of $z$ is $(1-q^{-1})|a-b|^{-1}$ so that the
total number of $v$ is $(1-q^{-1})(|1-a||a-b||b-1|)^{-1}$. Finally,
when $m = \ord_{\pi}$$ a = 0$, the constraints (1) and (2) are
equivalent to
\begin{itemize}
\item[(4)]  $ \ord_{\pi}$$ x(1-a) \ge 0,~~ \ord_{\pi}$$ z(a-b) \ge 0
~~$ and $~~ \ord_{\pi}$$(y(1-b)+xz(b-a)) \ge 0.$
\end{itemize}
Hence the numbers of $x$, $y$ and $z$ are $|1-a|^{-1}$, $|1-b|^{-1}$
and $|a-b|^{-1}$, respectively, so that the number of $v$ is
$(|1-a||a-b||b-1|)^{-1}$.

Since $\vol([\g])(|1-a||a-b||b-1|)^{-1}$ is present in both cases, it
suffices to compute
$$\frac{1}{\vol([\g])(|1-a||a-b||b-1|)^{-1}}\sum_{\kappa_\g(gK) \in [\g], ~ \text{type $1$}}
u^{l_A(\kappa_\g(gK))}.$$ In case $\ord_{\pi}$$ ~a > 0$, this sum is
equal to
$$u^{l_A([\g])}(1-q^{-1} + \sum_{m \ge 1}(1-q^{-1})^2q^{3m}u^{3m})
=u^{l_A([\g])}(1 - q^{-1})( \frac{1-q^2 u^3}{1-q^3u^3}),$$ and in
case $\ord_{\pi}$$ ~a = 0$, it is equal to
$$u^{l_A([\g])}(1 + \sum_{m \ge 1}(1-q^{-1})^2q^{3m}u^{3m})
=u^{l_A([\g])}\bigg(q^{-1} + (1 - q^{-1})( \frac{1-q^2
u^3}{1-q^3u^3})\bigg).$$ This proves the theorem.
\end{proof}

Contained in the proof above is the following statement.

\begin{corollary}\label{primitivesplit}
Suppose $\g \in \G$ is split with $r_\g = \diag(1, a, b)$. Assume
that $\g$ has type $1$ with $a \in \oo^\times$ . Let $\delta = \delta([\g]) = \ord_{\pi}$$
(1-a)$ and $n = \ord_{\pi}$$ b$. Then
\begin{eqnarray*}
 \Delta_A([\g]) = \{h v_x K
~|~ h \in  C_G(r_\g)/(C_G(r_\g) \cap K),  ~ v_x = \left(
\begin{matrix}
1 & x & \\
& 1 & \\
& & 1
\end{matrix} \right) ~\text{with} ~
x \in \pi^{-\delta} \oo/\oo \} \end{eqnarray*} and for $hv_xK \in
\Delta_A([\g])$, the geodesic $\kappa_\g(P_\g hv_xK)$ in $\B$ is
$$P_\g hv_xK \rightarrow P_\g hv_x \diag(1, 1, \pi)K \rightarrow
\cdots \rightarrow P_\g hv_x \diag(1, 1, \pi^n)K = \g P_\g hv_x K.$$
\end{corollary}

Here we used $P_\g hv_x \diag(1, 1, \pi^n)K = P_\g hv_x r_\g K =
P_\g r_\g h v_{ax}K = \g P_\g hv_x K$ since $v_{ax-x} \in K$.

\section{Homotopy cycles in $[\g]$ for $\g$ rank-one split}

In this section we fix a rank-one split $\g \in [\G]$ whose
eigenvalues $a, e+d\lambda, e + d \bar{\lambda}$, where $a, e, d \in
\oo$ and at least one of them is a unit, generate a quadratic
extension $L = F(\lambda)$ of $F$. Here $\lambda$ is a unit or
uniformizer in $L$ according as $L$ is unramified or ramified over
$F$, i.e., $\g$  is unramified or ramified rank-one split. Let
$r_\g = \left(
\begin{matrix} a &  &  \\ & e & dc \\ & d& e+db \end{matrix}\right)
$ be the rational form of $\g$ as in \S \ref{rationalform}. Fix a matrix $P_\g$ so that $P_\g^{-1}\g P_\g = r_\g$.

\subsection{The centralizers of $r_\g$ for $\g$ rank-one split} \label{centralizers}

 Embed $L^\times$ in $\GL_2(F)$ as the subgroup
\begin{eqnarray}\label{imbeddingL}
\bigg\{\left(\begin{matrix} u & v c\\
v & u + vb \end{matrix}\right) ~|~ u, v \in F, ~\text{not both zero}
\bigg\},
\end{eqnarray} which is further imbedded in $\GL_3(F)$ as
$\bigg\{\left(\begin{matrix} 1 & & \\ &u & v c\\
& v & u + vb \end{matrix}\right)\bigg\}.$ Embed $F^\times$ into
$\GL_3(F)$ as the diagonal matrices $\diag(F^\times, 1, 1)$. Note
that $r_\g$ lies in $F^\times \times L^\times$, and $F^\times \times
L^\times$ modulo the diagonal embedding of $F^\times$ in this
product is the centralizer of $r_\g$ in $G$. Recall from
(\ref{fundamentaldomain'}) that $C_{P_\g^{-1} \G
P_\g}(r_\g)\backslash C_G(r_\g)/(C_G(r_\g)\cap K)$ has cardinality
$\vol([\g])$.

 Observe that the group of units $\U_L$ of $L^\times$ is
contained in $K$. If $L$ is unramified over $F$, then $L^\times = \langle
\pi \rangle \U_L$ so that $C_G(r_\g)K/K $ is represented by the vertices
$\diag(\pi^n, 1, 1)K$, $n \in \mathbb Z$, on a line in $\B$, and
$C_{P_\g^{-1} \G P_\g}(r_\g)\backslash C_G(r_\g)/(C_G(r_\g)\cap K)$
 by  $\diag(\pi^n, 1, 1)K$, $n \mod
\vol([\g])$. If $L$ is ramified over $F$, then $L^\times = \langle \pi_L
\rangle \U_L$, where the uniformizer $\pi_L$ does not lie in $F$ and
$\pi_L^2$ differs from $\pi$ by a unit multiple. In this case
$C_G(r_\g) K/K $ is represented by the vertices $\diag(\pi^n, 1,
1)K$ and $\diag(\pi^n, 1, 1) \pi_L K$, $n \in \mathbb Z$, lying on
two lines in $\B$. There are two possibilities for $C_{P_\g^{-1} \G
P_\g}(r_\g)$:

Case (i). The vertices in $C_{P_\g^{-1} \G P_\g}(r_\g)K/K$ are
contained in the line $\diag(\pi^n, 1, 1)K$, $n \in \mathbb Z$. Then
$\vol([\g])$ is even so that $C_{P_\g^{-1} \G P_\g}(r_\g)\backslash
C_G(r_\g)/(C_G(r_\g)\cap K)$ is represented by the vertices
$\diag(\pi^n, 1, 1)K$ and $\diag(\pi^n, 1, 1) \pi_L K$, $n \mod
\vol([\g])/2$.

Case (ii). $C_{P_\g^{-1} \G P_\g}(r_\g)K/K$ contains a vertex on the
line $\diag(\pi^n, 1, 1) \pi_L K$, $n \in \mathbb Z$. Let $y \in
C_{P_\g^{-1} \G P_\g}(r_\g)$ be such that $yK = \diag(\pi^N, 1, 1)
\pi_L K$ has the least non-negative $N$. Then $y$ generates the
group $C_{P_\g^{-1} \G P_\g}(r_\g)$, $y^2K = \diag(\pi^{2N-1}, 1,
1)K$, $\vol([\g])= 2N-1$ is odd, and $C_{P_\g^{-1} \G
P_\g}(r_\g)\backslash C_G(r_\g)/(C_G(r_\g)\cap K)$ is represented by
the vertices $\diag(\pi^n, 1, 1)K$, $0 \le n \le N-1 =
(\vol([\g])-1)/2$, and $\diag(\pi^n, 1, 1)\pi_L K$, $0 \le n \le N-2 =
(\vol([\g])-3)/2$.

\subsection{Double coset representatives of $C_G(r_\g)\backslash G
/K$}

\begin{proposition}\label{rankonerep} The set
$$ S = \bigg\{ \left(\begin{matrix}1 & x & y \\ & 1 & 0\\ & & \pi^n
\end{matrix} \right) ~|~ x, y \in F/\oo, n \ge 0 \bigg\}$$
represents the double coset $C_G(r_\g)\backslash G /K$.
\end{proposition}
\begin{proof}
Write an element $g \in G$ as $wk$ for some upper triangular $w$ and
some $k \in K$. Since $C_G(r_\g) = F^\times \times L^\times$ modulo
the diagonal embedding of $F^\times$, we may assume that $w =
\left(\begin{matrix} 1 & x & y
\\ & 1 & z \\ &  & \pi^n
\end{matrix}\right)$, where $x, y, z \in F /\oo$ and $n \in \mathbb
Z$. We are reduced to proving
\begin{eqnarray}\label{Ldoublecoset}
 \GL_2(F) = \coprod_{n \ge 0} L^\times \left(\begin{matrix} 1 &
\\ & \pi^n \end{matrix} \right) GL_2(\oo),
\end{eqnarray}
where $L^\times$ is given by (\ref{imbeddingL}). The proof can be
found in \cite{Fl}, Lemma 1 on p.30.

\end{proof}

\subsection{Minimal lengths of cycles in $[\g]$}\label{minimallengthrankone}

First we discuss the type of $[\g]$, which is defined in \S4.4 to be
$(n,m)$ such that $r_\g \in T_{n,m} = K \diag(1, \pi^m,
\pi^{n+m})K$. Observe that $\ord_{\pi}$$ \det \g = \ord_{\pi}$$ \det
r_\g = \ord_{\pi}$$ a(e+d\lambda)(e+d \bar{\lambda}) \in 3\mathbb Z$
by assumption (I) on $\G$. Hence if $e+d\lambda$ is a unit in $L$,
then at least one of $e, d$ is a unit and $a$ is not a unit.
Consequently, $[\g]$ has type $(\ord_{\pi}$$ a, 0)$. Next assume
$e+d\lambda$ is not a unit. We distinguish two cases. If $L$ is
unramified over $F$ (hence $\lambda$ is a unit), then both $e$ and
$d$ are non-units and $a$ is a unit; in this case $[\g]$ has type
$(0, \min(\ord_{\pi}$$ e, \ord_{\pi}$$ d))$. If $L$ is ramified over $F$
(hence $\lambda$ is a uniformizer of $L$), then there are two
possibilities:

  (i) $\ord_{\pi}$$ (e+d\lambda)(e+d \bar{\lambda}) = 1$. This happens
if and only if $e$ is a non-unit, $d$ is a unit, and $\ord_{\pi}$$ a
\ge 2$; in this case
 $[\g]$ has type $(\ord_{\pi}$$ a - 1, 1)$.

 (ii) $\ord_{\pi}$$ (e+d\lambda)(e+d
\bar{\lambda}) > 1$. This happens if and only if both $e$ and $d$
are non-units and $a$ is a unit; in this case $[\g]$ has type $(0,
\ord_{\pi}$$ e)$ if $\ord_{\pi}$$ e \le \ord_{\pi}$$ d$, and type $(1,
\ord_{\pi}$$ d)$ if $\ord_{\pi}$$ e > \ord_{\pi}$$ d$.

This proves the first assertion of

\begin{theorem}\label{rankoneminlength}
Let $\g$ be a rank-one split element in $[\G]$ with rational form $r_\g = \left(
\begin{matrix} a &  &  \\ & e & dc \\ & d& e+db \end{matrix}\right) $.
Suppose that $r_\g \in K \diag(1, \pi^m, \pi^{m+n})K$. Then
\begin{itemize}
\item[(1)] The type $(n,~m)$ of $[\g]$ is as follows.
\begin{itemize}
\item[(1.i)] If $\ord_{\pi}$$ c = 0$, then $(n, m) = (\ord_{\pi}$$ a,
~\min\{\ord_{\pi}$$ e, ~\ord_{\pi}$$ d\})$.

\item[(1.ii)] If $\ord_{\pi}$$ c = 1$, then $(n, m) = (\ord_{\pi}$$ a,
~\ord_{\pi}$$ e)$ provided that $\ord_{\pi}$$ e \le \ord_{\pi}$$ d$,
otherwise $(n, m) = (\max\{\ord_{\pi}$$ a -1, ~1\}, ~\max\{\ord_{\pi}$$
d, ~1 \})$.
\end{itemize}

\item[(2)] $l_A([\g])= \min_{\kappa_\g(gK) \in [\g]}
l_A(\kappa_\g(gK)) = \ord_{\pi}$$ a(e^2 +edb - cd^2) = n+2m$.

 \item[(3)] $l_G([\g])=
  \min_{\kappa_\g(gK) \in
[\g]} l_G(\kappa_\g(gK)) = n+m$.
\end{itemize}
\end{theorem}

This theorem combined with Theorem \ref{minlength} completes the
proof of Theorem \ref{charlAandlG}.

\begin{remark} If $\g$ is ramified rank-one split and $[\g]$ has
type $(n, 1)$, then $[\g^2]$ has type $(2n+1, 0)$.
\end{remark}

\begin{proof} It remains to show that the algebraic and
geometric lengths of the cycles in $[\g]$ are at least those of
$[\g]$ since, as observed before, the cycles $\kappa_\g(P_\g gK)$
with $g \in C_G(r_\g)$
 have the same algebraic
and geometric lengths as $[\g]$. By Proposition \ref{rankonerep}, it
suffices to compute $(P_\g g)^{-1} \g P_\g g = g^{-1}r_\g g$ for $g
\in S$. Let $g = \left(\begin{matrix}1 & x & y
\\ & 1 & 0\\ & & \pi^i
\end{matrix} \right)$, where $x, y \in F/\oo$ and $i \ge 0$. Then

$$g^{-1}r_\g g
= \left(\begin{matrix} 1 & -x & -y\pi^{-i} \\ & 1 & 0 \\
 & & \pi^{-i}
\end{matrix} \right)
\left(
\begin{matrix} a &  &  \\ & e & dc \\ & d& e+db \end{matrix}\right)
\left(\begin{matrix}1 & x & y \\ & 1 & 0\\ & & \pi^i
\end{matrix} \right)$$
$$= \left(\begin{matrix}a & (a-e)x-dy\pi^{-i} & (a-e-db)y - cdx\pi^i \\ & e &
dc\pi^i\\ & d\pi^{-i} & e+db
\end{matrix} \right) \in K \left(\begin{matrix}\pi^{e_1} &  &  \\ & \pi^{e_2} & \\ & &
\pi^{e_3}
\end{matrix} \right)K.$$
Here $e_1 \le e_2 \le e_3$, and as in the proof of Theorem
\ref{minlength}, we have
\begin{eqnarray}\label{e1'}
e_1 \le \min \{\ord_{\pi} a, -i + \ord_{\pi} d, \ord_{\pi} e\} \le
\min\{\ord_{\pi} a, \ord_{\pi} d, \ord_{\pi} e \} = 0,
\end{eqnarray}
\begin{equation}\label{e1+e2'}
\begin{aligned}
e_1+e_2 &\le& \min \{\ord_{\pi} a e, ~ -i + \ord_{\pi} ad,
~\ord_{\pi}(e^2+bed-cd^2)\} \\
&\le& \min \{\ord_{\pi} a e, ~\ord_{\pi} ad,
~\ord_{\pi}(e^2+bed-cd^2)\} = m,
\end{aligned}
\end{equation} and
\begin{eqnarray}\label{e1+e2+e3'}
e_1+e_2+e_3 = \ord_{\pi} a(e^2+bed-cd^2)= n+2m,
\end{eqnarray}
in which the last upper bound for $e_1+e_2$ can be verified using
the statement (1). Therefore $l_A(\kappa_\g(P_\g gK)) = e_1+e_2+e_3
- 3 e_1 \ge e_1+e_2+e_3 = n+2m = l_A([\g])$ since $e_1 \le 0$. The
inequalities (\ref{e1+e2'}) and (\ref{e1+e2+e3'}) together give the
lower bounds $e_3 \ge n+2m - m = n+m$, which in turn implies
$l_G(\kappa_\g(P_\g gK)) = e_3-e_1 \ge n+m$. This proves the
theorem.
\end{proof}

As shown in the above proof, an algebraically tailless cycle in
$[\g]$ satisfies the condition $e_1 = 0$, while a tailless cycle
 in $[\g]$ should satisfy $e_1 + e_2 = m$ and
$e_1 = 0$. This shows that a tailless cycle is also
algebraically tailless. Moreover, it also satisfies $e_2 = m$, which
shows that a tailless cycle has the same type as $[\g]$. If
furthermore, $[\g]$ has type $1$, then an algebraically tailless
cycle in $[\g]$ satisfies $e_1 = 0$, which implies $e_1 + e_2 \ge 0$
and hence $e_1 + e_2 = 0 = m$ by (\ref{e1+e2'}) and $e_3 = n+m$.
This shows that in this case an algebraically tailless cycle in
$[\g]$ is also tailless. We record this discussion in

\begin{corollary}\label{rankonetailless}
Suppose $\g \in [\G]$ is rank-one split. Then all tailless cycles in
$[\g]$ are also algebraically tailless, and they have the same type
as $[\g]$. Moreover, if $[\g]$ has type $1$, then the algebraically
tailless and tailless cycles in $[\g]$ coincide.
\end{corollary}

We have shown that as long as $[\g]$ has type $1$, there is no
distinction between algebraically tailless and tailless, regardless
whether $\g$ is split or rank-one split.

\subsection{Counting the number of cycles in $[\g]$ in algebraic
length}

As observed before, for all $g \in C_{P_\g^{-1} \G
P_\g}(r_\g) \backslash C_G(r_\g)sK/K$, the cycles $\kappa_\g(P_\g gK)$ have
the same algebraic length. Since $S$ represents
the double coset $C_G(r_\g) \backslash G/K$, to count the number
of cycles in $[\g]$ of a given length,
we need to determine the cardinality of
$C_{P_\g^{-1} \G
P_\g}(r_\g) \backslash C_G(r_\g)sK/K$ for $s \in S$.  For this, we may take
as representatives the product of representatives of $C_{P_\g^{-1}
\G P_\g}(r_\g) \backslash C_G(r_\g)/(C_G(r_\g)\cap K)$ (independent
of $s$) by the representatives of $(C_G(r_\g)\cap K)sK/K$.
The number of the former representatives is $\vol([\g])$, defined by
(\ref{fundamentaldomain'}).

It remains
to compute the cardinality of the latter. Recall that $L^\times
\cap K$ consists of the units in $L^\times$, which we shall
identify as the set of matrices
\begin{eqnarray*}
\U_L = \bigg\{\left(\begin{matrix} u & v c\\
v & u + v b \end{matrix}\right) ~|~ u, v \in \oo, u^2 + uvb - cv^2
~\text{is a unit} \bigg\}.
\end{eqnarray*}
Denote by $K'$ the group $GL_2(\oo)$. As analyzed in the proof of
Proposition \ref{rankonerep}, we are reduced to counting, for
given $m \ge 0$, the cardinality of $\U_L \left(\begin{matrix} 1 &
\\  & \pi^m \end{matrix}\right)K'/K'$.

\begin{proposition}\label{Cgammadoublecosets}
$$\#[\U_L
\left(\begin{matrix} 1 &
\\  & \pi^m \end{matrix}\right)K'/K'] =
\left\{
\begin{array}{lcl}
1&& {\,when\,}~m=0,\\
q^m && {\,when\,}~m \ge 1 ~{\, and \,}~\ord_\pi c=1, \\
q^m+q^{m-1}&& {\,when\,}~m \ge 1 ~{\,and\,}~\ord_\pi c=0. \\
\end{array}
\right.
$$
\end{proposition}
\begin{proof} It is clear that the cardinality is $1$ when $m =
0$. Thus assume $m \ge 1$.

Case (I) $\ord_\pi$$ c=1 $. Then any $\left(\begin{matrix} u & v c\\
v & u + v b \end{matrix}\right) \in \U_L$ satisfies $u \in
\oo^\times$. For $n \ge 0$, let
$$\U_L(n) = \bigg\{\left(
\begin{matrix} u &  vc\pi^n \\  v\pi^n & u +vb\pi^n \end{matrix}\right) \in \U_L \bigg| u,v \in \oo^\times \bigg\}$$
so that
$$\U_L = \U_L(\infty) \cup_{n \ge 0} \U_L(n),$$
where
$$\U_L(\infty) = \bigg\{\left(\begin{matrix} u & 0\\
0 & u \end{matrix}\right) ~|~u \in \oo^\times \bigg\}.$$
 One verifies that
$$\U_L(n)\left(
\begin{matrix} 1 &   \\  & \pi^m \end{matrix}\right)K' = \bigcup_{u
\in \oo^\times/\pi^{m-n}\oo}\left( \begin{matrix} \pi^{m-n} & u  \\
& \pi^n
\end{matrix}\right)K'$$
for $0 \le n < m$, and
$$\U_L(n)\left(
\begin{matrix} 1 &   \\  & \pi^m \end{matrix}\right)K' = \left(
\begin{matrix} 1 &   \\  & \pi^m \end{matrix}\right)K'$$
for $n \ge m$ and $n = \infty$. Therefore
$$ \#[\U_L
\left(\begin{matrix} 1 &
\\  & \pi^m \end{matrix}\right)K'/K'] = 1 + \sum_{0 \le n < m}(q^{m-n}-q^{m-n-1})=q^m.$$

Case (II) $\ord_\pi$$ c=0 $. Let
$$\U_L' = \bigg\{\left(
\begin{matrix} u &  vc \\  v & u + vb \end{matrix}\right) \in \U_L \bigg| u \in \oo^\times
\bigg\}$$ and $$\U_L'' =  \bigg\{\left(
\begin{matrix} u  &  vc \\  v & u +vb \end{matrix}\right) \in \U_L \bigg| u \in \pi\oo \bigg\}
$$
so that
$$\U_L = \U_L' \cup \U_L''.$$
As in Case (I), we have
$$ \U_L'\left(
\begin{matrix} 1 &   \\  & \pi^m \end{matrix}\right)K' =
\bigcup_{\substack{ m\geq n\geq 0 \\u \in \oo^\times/\pi^{m-n}\oo}} \left( \begin{matrix} \pi^{m-n} & u  \\
& \pi^n
\end{matrix}\right)K'.$$
One checks that
$$ \U_L''\left(
\begin{matrix} 1 &   \\  & \pi^m \end{matrix}\right)K' =
\bigcup_{z \in \pi\oo/\pi^m \oo} \left( \begin{matrix} \pi^{m} & z  \\
& 1
\end{matrix}\right)K'.
$$
Therefore $$ \#[\U_L \left(\begin{matrix} 1 &
\\  & \pi^m \end{matrix}\right)K'/K']= q^m+q^{m-1}$$
for $m \ge 1$.
\end{proof}

We summarize the above discussion in

\begin{corollary}\label{rankonedoublecosets}
For each $ s = \left(\begin{matrix}1 & x & y \\ & 1 & 0\\ & &
\pi^n
\end{matrix} \right) \in S$, the cardinality of
$C_{P_\g^{-1} \G P_\g}(r_\g) \backslash C_G(r_\g)sK/K$ is
$$
\vol([\g])\left\{
\begin{array}{lcl}
1&& {\,when\,}~n=0,\\
q^n && {\,when\,}~n \ge 1 ~{\, and \,}~\ord_\pi c=1, \\
q^n+q^{n-1}&& {\,when\,}~n \ge 1 ~{\,and\,}~\ord_\pi c=0. \\
\end{array}
\right.
$$
\end{corollary}

Now we are ready to state the main result of this section.

\begin{theorem}\label{numberinrankoneclass}
Suppose $\g \in [\G]$ is rank-one split with rational form $r_\g = \left(
\begin{matrix} a &  &  \\ & e & dc \\ & d& e+db \end{matrix}\right) $. Set
$\de = \de([\g])= \ord_{\pi}$$ d$ and $\mu = \mu([\g])= \ord_{\pi}$$
((a-e)^2 - db(a-e) - cd^2)$.
\begin{itemize}
\item[(A)] $\g$ is unramified rank-one split. Then the following hold.
\begin{itemize}
\item[(A1)]
$$ \sum_{\kappa_\g(gK) \in [\g]} u^{l_A(\kappa_\g(gK))}
= \vol([\g])u^{l_A([\g])}\bigg(\frac{q^{\de+1}+q^{\de}-2}{q-1}+
\frac{(q+1)q^{\de+2}u^3}{1-q^3 u^3}\bigg) \bigg(\frac{1-u^3}{1-q^2
u^3}\bigg).$$

\item[(A2)] If $[\g]$ does not have type $1$, then
$$ \sum_{\kappa_\g(gK) \in [\g], ~\text{type $1$}} u^{l_A(\kappa_\g(gK))}=
\vol([\g])u^{l_A([\g])}\bigg(q^{\de} + q^{\de - 1}+
\frac{(q^2-1)q^{\de +1}u^3}{1-q^3u^3} \bigg).$$

\item[(A3)] If $[\g]$ has type $1$, then
$$ \sum_{\kappa_\g(gK) \in [\g], ~\text{type $1$}} u^{l_A(\kappa_\g(gK))}=
\vol([\g])u^{l_A([\g])}\bigg(\frac{q^{\de+1}+q^{\de}-2}{q-1}+
\frac{(q^2-1)q^{\de +1}u^3}{1-q^3u^3} \bigg).$$
\end{itemize}

\item[(B)] $\g$ is ramified rank-one split. Then the
following hold.
\begin{itemize}
\item[(B1)]
$$ \sum_{\kappa_\g(gK) \in [\g]} u^{l_A(\kappa_\g(gK))} =
\vol([\g])q^{\mu}u^{l_A([\g])}\bigg(\frac{q^{\de +1} -1}{q-1} +
\frac{q^{\de +3}u^3}{1-q^3u^3} \bigg)\frac{1-u^3}{1-q^2u^3}.$$

\item[(B2)] If $[\g]$ does not have type $1$, then
$$ \sum_{\kappa_\g(gK) \in [\g], ~\text{type $1$}} u^{l_A(\kappa_\g(gK))}=
\vol([\g])u^{l_A([\g])}\bigg(q^{\de}(q^{\mu} - \mu) +
\frac{(q-1)q^{\de + \mu +2}u^3}{1-q^3u^3}\bigg).$$

\item[(B3)] If $[\g]$ has type $1$, then
$$ \sum_{\kappa_\g(gK) \in [\g], ~\text{type $1$}} u^{l_A(\kappa_\g(gK))}=
\vol([\g])u^{l_A([\g])}\bigg(\frac{q^{\de +1} -1}{q-1} +
\frac{(q-1)q^{\de +2}u^3}{1-q^3u^3}\bigg).$$
\end{itemize}
\end{itemize}
Moreover, in each case, if $[\g]$ does not have type $1$, none of
the type $1$ cycles in $[\g]$ are tailless.
\end{theorem}

\begin{remarks}

1. $\mu = 0$ unless $a, e, c$ are all nonunit, in which case it is
$1$ and $\de = 0$.

2. $\mu = 0$ when $[\g]$ has type one.

3. $\de > 0$ in case (A2), while $\de$ may be zero in case (A3).

4. The right hand side of the identities (A1) - (B3) can be expressed as $\vol([\g])$
times the orbital integrals at the rank-one split
element $\g$ of suitably chosen spherical functions on $G$ with fast decay.
\end{remarks}

\begin{proof}
Recall that the algebraic length of a cycle in $[\g]$ is equal to
$l_A([\g]) + 3m$ for some $m \ge 0$. We shall follow the same
notation and computation as in the proof of Theorem
\ref{rankoneminlength}, letting $g$ run through all elements in the
double coset representatives $S$ and computing, for each $m \ge 0$,
the number of cycles $\kappa_\g(P_\g gK)$ with $l_A(\kappa_\g(P_\g
gK)) \le l_A([\g]) + 3m$ using
Corollary \ref{rankonedoublecosets}. As $g = \left(\begin{matrix}1 & x & y \\ & 1 & 0\\
& & \pi^i
\end{matrix} \right)$, this amounts to computing the number of $x, y
\in F/\oo$ and $i \ge 0$ such that
\begin{eqnarray*}
e_1 = \min \{\ord_{\pi} ((a-e)x -d\pi^{-i}y) , ~\ord_{\pi} (-
cd\pi^ix + (a-e-db)y ),
 -i+\ord_{\pi}d \}\ge -m.
\end{eqnarray*} This is equivalent to $0 \le i \le m+\ord_{\pi}$$d $, $(a-e)x -d\pi^{-i}y
 \in \pi^{-m}\oo$ and $-cd\pi^ix + (a-e-db)y \in \pi^{-m}\oo$.
 Denote $\ord_\pi$$ d$ by $\de$ for short. So for each $0
\le i \le m+  \de $, we solve the following system of linear
equations
\begin{eqnarray}\label{system1}
\left(\begin{matrix}\alpha\\
\beta
\end{matrix}\right)
=
\left(\begin{matrix} a-e & -d\pi^{-i}\\
-cd\pi^i & a-e-db
\end{matrix}\right)
\left(\begin{matrix}x\\
y
\end{matrix}\right)=M \left(\begin{matrix}x\\
y
\end{matrix}\right)
\end{eqnarray}
for $\alpha, \beta \in \pi^{-m}\oo$ and count the distinct pairs
$(x, y) \in F/\oo \times F/\oo$. Recall that $a, e, d$ are
integral, at least one of them is a unit, and $a$ and $e$ cannot
be both units since $\ord_{\pi}$$ \det r_\g > 0$. Let
$$\mu := \ord_{\pi} \det M = \ord_{\pi} ((a-e)^2 - db(a-e) -
cd^2),$$ which is 0 unless $a$, $e$ and $c$ are all nonunits, in
which case it is $1$. Put
$$\ve := \min \{\ord_{\pi} (a-e), -i + \de, \ord_{\pi}
(a-e-bd)\},$$ which is equal to $-i + \de$ if $\de \le i \le m+\de$,
and 0 if $0 \le i < \de$. Then the coefficient matrix $M = k_1
\diag(\pi^{\ve}, \pi^{\mu-\ve})k_2$ for some $k_1, k_2 \in
GL_2(\oo)$. Thus system (\ref{system1}) has the same number of
solutions as the system
\begin{eqnarray}\label{system2}
\left(\begin{matrix}\alpha\\
\beta
\end{matrix}\right)
=
\left(\begin{matrix} \pi^{\ve} &  \\
  & \pi^{\mu -\ve}
\end{matrix}\right)
\left(\begin{matrix}x\\
y
\end{matrix}\right)
\end{eqnarray}
for $\alpha, \beta \in \pi^{-m}\oo$ and $(x, y) \in F/\oo \times
F/\oo$. We get the solutions $x \in \pi^{-m -\ve}\oo/\oo$ and $y \in
\pi^{-m -\mu +\ve}\oo/\oo$ so that there are $q^{2m+\mu}$ different
pairs $(x, y)$ for each $0 \le i \le m+  \de $. To proceed, we
distinguish two cases.

Case (A) $\ord_{\pi}$$ c = 0$, that is, $\g$ is unramified rank-one
split. Then $\mu = 0$. By Corollary \ref{rankonedoublecosets}, the
number of classes in $[\g]$ with algebraic length at most
$l_A([\g])+3m$ is
\begin{eqnarray*}
\vol([\g])q^{2m}(1 +
 \sum_{1 \le n \le m+\de} q^n + q^{n-1}) &=&
\vol([\g])q^{2m}(\frac{q^{m+\de}-1}{q-1} + \frac{q^{m+\de+1}-1}{q-1})\\
&=& \frac{\vol([\g])}{q-1}(q^{3m+\de+1}+ q^{3m+\de} -2q^{2m}).
\end{eqnarray*}
Therefore
\begin{eqnarray*}
& & \sum_{\kappa_\g(gK) \in [\g]} u^{l_A(\kappa_\g(gK))} =
\sum_{\kappa_\g(P_\g gK) \in [\g]} u^{l_A(\kappa_\g(P_\g gK))} \\
&=&\vol([\g])u^{l_A([\g])}\frac{1}{q-1}\bigg(q^{\de+1} + q^{\de} -2 +\\
& &~~~~~~~~\sum_{m
\ge 1}(q^{3m+\de+1}+q^{3m+\de}-2q^{2m} -q^{3m+\de-2}-q^{3m+\de-3}+2q^{2m-2}) u^{3m}\bigg) \\
&=& \vol([\g])u^{l_A([\g])}\frac{1}{q-1}\bigg(\frac{q^{\de+1} +
q^{\de}}{1 - q^3u^3} - \frac{2}{1-q^2u^3}\bigg)(1-u^3)\\
&=& \vol([\g])u^{l_A([\g])}\bigg(\frac{q^{\de+1}+q^{\de}-2}{q-1}+
\frac{(q+1)q^{\de+2}u^3}{1-q^3 u^3}\bigg) \bigg(\frac{1-u^3}{1-q^2
u^3}\bigg).
\end{eqnarray*}

Among the cycles with $l_A(\kappa_\g(P_\g gK)) = l_A([\g]) + 3m$, we
compute the number of those with type $1$.  First consider the case
$m \ge 1$. In order that $l_A(\kappa_\g(P_\g gK)) = l_A([\g]) + 3m$
and $\kappa_\g(P_\g gK)$ has type $1$, two conditions must be
satisfied:
$$
e_1 = \min \{\ord_{\pi} ((a-e)x -d\pi^{-i}y) , ~\ord_{\pi} (-
cd\pi^ix + (a-e-db)y ),
 -i+\de \}= -m, $$ and
\begin{eqnarray*}
e_1+e_2 = \ord_{\pi} [((a-e)x -d\pi^{-i}y)(e+db) - d\pi^{-i}(-
cd\pi^ix + (a-e-db)y)] = -2m.
\end{eqnarray*}
These two conditions are equivalent to $i = \de +m$, $ \ord_{\pi}$$
(-cd\pi^ix + (a-e-db)y ) = -m,$ and $\ord_{\pi}$$ ((a-e)x
-d\pi^{-i}y)\ge -m.$ This amounts to solving system (\ref{system1})
with $\alpha \in \pi^{-m}\oo$ and $\beta \in \pi^{-m}\oo^\times$,
hence we obtain $(q-1)q^{2m-1}$ distinct pairs $(x, y)$. Combined
with Corollary \ref{rankonedoublecosets}, we see that the number of
rank one cycles $\kappa_\g(P_\g gK)$ with $l_A(\kappa_\g(P_\g
gK))=l_A([\g]) + 3m$ is $\vol([\g])(q-1)q^{2m-1}(q^{\de + m} + q^{\de +
m -1})$.

Next consider the case $m=0$. Under the assumption $\ord_{\pi}$$ c =
0$, we know from Theorem \ref{rankoneminlength} that $[\g]$ has
type $(\ord_{\pi}$$ a, \min\{\ord_{\pi}$$ e, \ord_{\pi}$$ d\})$.
Therefore it has type $1$ if and only if $\ord_{\pi}$$ a
> 0$, in which case all cycles in $[\g]$
with algebraic length equal to $l_A([\g])$ have type $1$, and the
number of such cycles is $\vol([\g])\frac{q^{\de+1}+q^{\de}-2}{q-1}$,
as computed above. If $[\g]$ does not have type $1$, then $\de =
\ord_{\pi}$$ d > 0$; the condition $e_1 = e_2 = 0$ implies $i = \de$
and only one solution $(x, y)=(0, 0)$. In this case the number of
type $1$ cycles in $[\g]$ with algebraic length equal to $l_A([\g])$
is $q^{\de} + q^{\de - 1}$ by Corollary \ref{rankonedoublecosets}.
Put together, we have shown the following:

If $[\g]$ has type $1$, then
\begin{eqnarray*}
\sum_{\kappa_\g(gK) \in [\g], ~\text{type $1$}}
u^{l_A(\kappa_\g(gK))} &=&
\vol([\g])u^{l_A([\g])}\bigg(\frac{q^{\de+1}+q^{\de}-2}{q-1}+ \sum_{m
\ge 1}(q-1)q^{2m-1}(q^{\de + m} + q^{\de + m -1})u^{3m} \bigg) \\
&=& \vol([\g])u^{l_A([\g])}\bigg(\frac{q^{\de+1}+q^{\de}-2}{q-1}+
\frac{(q^2-1)q^{\de +1}u^3}{1-q^3u^3} \bigg),
\end{eqnarray*}
while if $[\g]$ does not have type $1$, then
\begin{eqnarray*}
 \sum_{\kappa_\g(gK) \in [\g], ~\text{type $1$}}
u^{l_A(\kappa_\g(gK))} = \vol([\g])u^{l_A([\g])}\bigg(q^{\de} + q^{\de
- 1}+ \frac{(q^2-1)q^{\de +1}u^3}{1-q^3u^3} \bigg).
\end{eqnarray*}

Case (B) $\ord_{\pi}$$ c = 1$, that is, $\g$ is ramified rank-one
split. Then $\mu = 0$ or $1$. The same computation as in Case (A)
together with Corollary \ref{rankonedoublecosets} shows that the
number of classes in $[\g]$ with algebraic length at most
$l_A([\g])+3m$ is
$$\vol([\g])q^{2m+\mu}\sum_{0 \le n \le m+  \de} q^n =
\vol([\g])q^{2m+\mu}\frac{q^{m+\de +1} - 1}{q-1} =
\vol([\g])\frac{q^{\mu}}{q-1}(q^{3m+\de +1} - q^{2m}).$$ Therefore
\begin{eqnarray*}
\sum_{\kappa_\g(gK) \in [\g]} u^{l_A(\kappa_\g(gK))}
&=&\vol([\g])\frac{q^{\mu}}{q-1}u^{l_A([\g])}\bigg(\sum_{m \ge
0}(q^{3m+\de +1} - q^{2m})u^{3m} - \sum_{m \ge
1}(q^{3m+\de -2} - q^{2m-2})u^{3m}\bigg) \\
&=& \vol([\g])\frac{q^{\mu}}{q-1}u^{l_A([\g])}\bigg(\frac{q^{\de
+1}}{1-q^3u^3} - \frac{1}{1-q^2u^3}\bigg)(1-u^3) \\
&=& \vol([\g])q^{\mu}u^{l_A([\g])}\bigg(\frac{q^{\de +1} -1}{q-1} +
\frac{q^{\de +3}u^3}{1-q^3u^3} \bigg)\frac{1-u^3}{1-q^2u^3}.
\end{eqnarray*}

 Now we compute the number of type $1$ cycles $\kappa_\g(P_\g gK)$ with
  algebraic length $l_A(\kappa_\g(P_\g gK)) = l_A([\g]) + 3m$.
First consider the case $m \ge 1$. Following the same argument as in
Case (A) and applying Corollary \ref{rankonedoublecosets}, we see
that the number of such cycles is $\vol([\g])(q-1)q^{2m +\mu -1}q^{\de
+ m}$.

Next we discuss the remaining case $m = 0$. By Theorem
\ref{rankoneminlength}, $[\g]$ has type one if and only if
$\ord_{\pi}$$ a > 0$ and $\ord_{\pi}$$ e = 0$, in which case all cycles
in $[\g]$ with algebraic length equal to $l_A([\g])$ are of type
one, and the number of such cycles is $\vol([\g])q^{\mu}\frac{q^{\de
+1} -1}{q-1}$. When $[\g]$ does not have type $1$, we have
$\ord_{\pi}$$ e > 0$; the condition $e_1 = e_2 = 0$ implies $i = \de$.
Moreover, if $\mu= 0$, in which case $a$ is a unit, then there is
only one pair $(x, y) = (0, 0)$; while if $\mu = 1$, in which case
$a$ is not a unit, then there are $q-1$ pairs $(x, y) = (0, y)$ with
$ y \in \pi^{-1}\oo^\times/\oo$ so that $\ord_{\pi}$$ (-cd\pi^ix +
(a-e-db)y ) = 0$. Consequently, when $[\g]$ does not have type $1$,
the number of type $1$ cycles in $[\g]$ with algebraic length equal
to $l_A([\g])$ is $\vol([\g])q^{\de}$ if $\mu = 0$, and
$\vol([\g])(q-1)q^{\de}$ if $\mu = 1$. In other words, it is
$\vol([\g])q^{\de}(q^{\mu} - \mu)$. Summing up, we have proved the
following:

If $[\g]$ has type $1$, then
\begin{eqnarray*}
\sum_{\kappa_\g(gK) \in [\g], ~\text{type $1$}}
u^{l_A(\kappa_\g(gK))} &=&
\vol([\g])u^{l_A([\g])}q^{\mu}\bigg(\frac{q^{\de +1} -1}{q-1} +
\sum_{m \ge 1}(q-1)q^{3m +\de -1}u^{3m}\bigg) \\
&=& \vol([\g])u^{l_A([\g])}q^{\mu}\bigg(\frac{q^{\de +1} -1}{q-1} +
\frac{(q-1)q^{\de +2}u^3}{1-q^3u^3}\bigg),
\end{eqnarray*}
while if $[\g]$ does not have type $1$, then
\begin{eqnarray*}
\sum_{\kappa_\g(gK) \in [\g], ~\text{type $1$}}
u^{l_A(\kappa_\g(gK))} &=& \vol([\g])u^{l_A([\g])}\bigg(q^{\de}(q^{\mu}
- \mu) + \frac{(q-1)q^{\de + \mu +2}u^3}{1-q^3u^3}\bigg).
\end{eqnarray*}
This completes the proof of the theorem.
\end{proof}

As before, let
\begin{eqnarray}
\Delta_A([\g]) = \{ gK \in G/K ~|~ l_A(\kappa_\g(P_\g gK)) =
l_A([\g]) \}.
\end{eqnarray}
Then $\Delta_A([\g])$ contains $C_G(r_\g)K/K$ and it is invariant
under left multiplication by $C_{P_\g^{-1} \G P_\g}(r_\g)$.
Moreover, $C_{P_\g^{-1} \G P_\g}(r_\g)\backslash \Delta_A([\g])$ is
finite, and its cardinality is the number of algebraically tailless
cycles in $[\g]$. Contained in the proofs of Corollary
\ref{rankonedoublecosets} and Theorem \ref{numberinrankoneclass} is
the first assertion of the proposition below. Let
\begin{eqnarray}\label{ginuandgiz}
g_{i,j,u} = \left(
\begin{matrix}
1 &  & \\
& \pi^{i-j} & u\\
& & \pi^j
\end{matrix} \right) \quad \text{and} \quad g_{i,z} = \left(
\begin{matrix}
1 &  & \\
& \pi^{i} & z\\
& & 1
\end{matrix} \right).
\end{eqnarray}

\begin{proposition}\label{rankonenumberofalgtailless}
Let $\g \in [\G]$ be rank-one split with $r_\g = \left(
\begin{matrix} a &  &  \\ & e & dc \\ & d& e+db \end{matrix}\right) $. Set
$\de = \de([\g])= \ord_{\pi}$$ d$. Suppose that $[\g]$ has type $1$
with $ n = \ord_\pi$$ a$.
Then
\begin{eqnarray*}
\Delta_A([\g]) &=& \{h g_{i,j,u} K ~|~ h \in  C_G(r_\g)/(C_G(r_\g)
\cap K), ~0 \le j \le i
\le \delta, \\
& & ~~u \in \oo^\times/\pi^{i-j}\oo ~\text{for} ~j < i,
  ~\text{and} ~ u = 0 ~\text{for} ~j = i\}
\end{eqnarray*} if $\g$ is ramified rank-one split, and
\begin{eqnarray*}
\Delta_A([\g]) = \{h g_{i,j,u} K~|~ h ~\text{and} ~g_{i,j,u}
~\text{as above}\} ~\cup ~\{hg_{i,z}K ~|~ h ~\text{as above}, ~1 \le
i \le \delta, z \in \pi \oo/\pi^i \oo \}
\end{eqnarray*}
if $\g$ is unramified rank-one split. Consequently, the number of
algebraically tailless cycles in $[\g]$ is
\begin{eqnarray*}
\#(C_{P_\g^{-1} \G P_\g}(r_\g)\backslash \Delta_A([\g])) =
\vol([\g])\left\{
\begin{array}{ll}
\frac{q^{\de+1}+q^{\de}-2}{q-1} & \mbox{if $[\g]$ is unramified rank-one split}, \\
\frac{q^{\de +1} -1}{q-1}& \mbox{if $[\g]$ is ramified rank-one
split}.
\end{array}
\right.
\end{eqnarray*}
Moreover, for $g = h g_{i,j,u}$ or $hg_{i,z}$ such that $gK \in
\Delta_A([\g])$, the geodesic $\kappa_\g(P_{\g}gK)$ in $\B$ is given
by $P_{\g}gK \rightarrow P_{\g}g \diag(\pi, 1, 1)K \rightarrow
\cdots \rightarrow P_{\g}g \diag(\pi^n, 1, 1)K = \g P_{\g}gK.$
\end{proposition}

The last assertion follows from $P_{\g}g \diag(\pi^n, 1, 1)K =
P_{\g}g r_\g K = P_{\g}r_\g gK = \g P_{\g}gK$ since $g^{-1} r_\g g
\in K$ by choice.

\subsection{Tailless type $1$ primitive cycles}\label{primitivecycles}

A cycle in $X_\G$ is {\it primitive} if
it is not obtained by repeating a cycle more than once. Suppose that $\kappa_\g(gK)$ is
$\kappa_{\beta}(gK)$ repeated $m$ times in $X_\G$. Then $\g gK =
\beta^m gK$ in $\B$. As the action of $\G$ on $\B$ is fixed point
free, this implies $\g = \beta^m$. In other words, a necessary
condition for a cycle $\kappa_\g(gK)$ to be non-primitive is that
$\g$ is a positive power of a non-identity element in $\G$.

On the other hand, suppose $\g \in \G$ is of type $1$ and $\g = \beta^m$ for a
unique $\beta \in \G$ and $m > 1$. Then $r_\g$ and
$r_{\beta}$ have the same centralizers in $\G$, and $vol([\g]) =
vol([\beta^j])$ for all $j \ge 1$ (cf. \S \ref{volume}). Moreover, $\beta^j$ also has type $1$ and $\delta([\beta^j]) \le \delta([\g])$ for all positive
divisors $j$ of $m$. Combining Corollary \ref{primitivesplit} and
Proposition \ref{rankonenumberofalgtailless}, we conclude that
$\Delta_A([\beta^j]) \subset \Delta_A([\g])$ for $j|m$, and the
cycles $\kappa_\g(gK)$ with $gK$ in $\Delta_A([\g]) \smallsetminus
\cup_{j|m, ~0<j<m} ~\Delta_A([\beta^j])$ are the tailless type $1$
primitive closed geodesics in $[\g]$. Further, by shifting vertices
on such a cycle we obtain $l_A([\g])$ distinct cycles.

This is different from the case of graphs arising from $\PGL_2(F)$.
See more discussions about this at the end of \S10.

\section{Gallery Zeta function of $X_\G$}

\subsection{Chambers and Iwahori-Hecke algebra on the building $\B$}

A chamber of the building $\B = G/K$ is a $2$-simplex with three
mutually adjacent vertices $v_1, v_2, v_3$. The group $G$ acts on
the vertices of $\B$ transitively, and it preserves edges and
chambers of $\B$. Let
$$\sigma =\left( \begin{matrix}  & 1 &  \\
& & 1 \\ \pi & & \end{matrix}\right).$$ Denote by $C_0$ the
fundamental chamber with vertices $v_1= K, v_2=\sigma K$, and
$v_3=\sigma^2 K$. The Iwahori subgroup $B$ of $K$ consisting of
elements $k \in K$ congruent to upper triangular matrices mod $\pi$
is the largest subgroup of $G$ stabilizing each vertex of $C_0$,
while $\sigma$ rotates the vertices of $C_0$. Denote by $\sigma'$
the permutation $(1~2~3)$ in $S_3$ such that $\sigma(v_i) =
v_{\sigma'(i)}$. Since $G$ acts transitively on the chambers of
$\B$, the assignment $gB \mapsto gC_0$ is a three-to-one map from
$G/B$ to the set of chambers such that $gB$, $g\sigma B$ and $g
\sigma^2 B$ all correspond to the same chamber $gC_0$. The matrices
$$t_1=\left( \begin{matrix}  & 1 & \\
 1& &  \\\ & & 1\end{matrix}\right), \quad t_2= \left( \begin{matrix}  &  & \pi^{-1} \\
 &1 &  \\\pi & & \end{matrix}\right), \quad  \quad \text{and} \quad
 t_3=\left( \begin{matrix}  1 &  & \\
 &  & 1 \\\ & 1 &  \end{matrix}\right)$$
act as reflections which fix the edges $\{v_1, v_2\}$, $\{v_2,
v_3\}$ and $\{v_3, v_1\}$ of $C_0$, respectively. We have $\sigma
t_i = t_{\sigma'(i)}\sigma$ for $i = 1, 2, 3$.

Note that $t_1, t_2, t_3$ generate the Weyl group $W$ of
$\PSL_3(F)$ subject to the relations $t_i^2 = Id$ and $(t_i t_j)^3
= Id$ for $i \ne j$.
The Bruhat decomposition of $G$ is
$$ G= \coprod_{w \in W \ltimes \langle \sigma\rangle}  B w B. $$
Each element $w \in W \ltimes \langle \sigma \rangle$ defines an operator $L_{w}$
on $L^2(G/B)$ by sending a function $f$ to $L_{w}f$ given by
$$ L_{w} f (gB) = \sum_{w_i B \in BwB/B} f(g w_i B)\quad \qquad \text{for all} ~gB.$$
These operators form a generalized Iwahori-Hecke algebra satisfying
the following relations (cf. \cite{Ga}):
\begin{itemize}
\item[1.] $L_{t_i}\cdot L_{t_i}= (q-1) L_{t_i} + q Id$,

\item[2.] $L_{t_i}\cdot L_{t_j}=  L_{t_i t_j}$ for $i \ne j$,

\item[3.] $L_{t_i}\cdot L_w = L_{t_i w}$ if the length of $t_iw$
is $1$ plus the length of $w$,

 \item[4.]
$L_{\sigma} \cdot L_{t_i}= L_{ \sigma t_i }= L_{
t_{\sigma'(i)}\sigma} $ for $i = 1, 2, 3$.
\end{itemize}

\noindent Let
\begin{eqnarray}\label{LB}
L_B = L_{t_2 \sigma^2}.
\end{eqnarray} Then the above properties
imply $(L_B)^{3n} = (L_{t_2t_1t_3})^n$ for $n \ge 1$.

\subsection{Galleries in $\B$} Paths formed by the edge-adjacent chambers are called
galleries. A geodesic gallery between two chambers is a gallery
containing the least number of intermediate chambers.
To get geodesic galleries from $g_1B$ to $g_2B$, we find the element
$w \in W \ltimes\langle \sigma \rangle$ such that $g_1^{-1} g_2 \in BwB$ and
write $w = t_{i_1}\cdots t_{i_n}\sigma^j$ as a word using least
number of reflections $t_1, t_2, t_3$; call $n$ the {\it length} of
the gallery. All geodesic galleries from $g_1B$ to $g_2B$ have
length $n$; different galleries arise from different expressions of
$w$ as a product of generators, and they are regarded as {\it
homotopic}. Like the case of paths, given two distinct chambers
$g_1B$ and $g_2B$, there is only one homotopic class of geodesic
galleries in $\B$ from $g_1B$ to $g_2B$.

Observe that a geodesic gallery arising from $w = t_{i_1}\cdots
t_{i_n}\sigma^j$ is a strip if and only if the difference $i_k -
i_{k+1}$ remains the same mod 3 for $ 1 \le k \le n-1$. It is said
to have type $1$ or $2$ according to the common difference being $1$
or $2$. Note that the homotopy class of a gallery of type $1$ or $2$
contains only one geodesic gallery, thus we shall drop the word
"homotopy" in this case.

\subsection{Closed galleries in $X_\G$} A closed gallery in $X_\G$
starting at the chamber $\G gB$ of $X_\G$ can be lifted to a gallery
in $\B$ starting at $gB$ and ending at $\g gB$ for some $\g \in \G$.
Denote by $\ka_\g(gB)$ the homotopy class of geodesic galleries in
$\B$ from $gB$ to $\g gB$. By abuse of notation, it also represents
a homotopy class of closed geodesic gallery in $X_\G$ starting at
$\G gB$.
The
 argument in \S \ref{parametrization} holds with $K$ replaced by $B$.
 Let, for $\g \in [\G]$,
 $$ [\g]_B = \{\ka_\g(gB)~:~ g \in C_\G(\g) \backslash G /B \}.$$
 Then the union of $[\g]_B$ over $\g \in [\G]$ is the set of all
  vertex-based homotopy classes of closed geodesic galleries
 in $X_\G$.

 A closed gallery $\ka_\g(gB)$ of length $n$ in $X_\G$ is called {\it tailless} if the
 geodesic gallery in $\B$ from $gB$ to $\g gB$ followed by the geodesic gallery
 from $\g gB$ to $\g^2 gB$ is a
 geodesic gallery from $gB$ to $\g^2 gB$ of length $2n$.  Note that the condition (I)
 imposed on $\G$ in \S\ref{rationalform} implies that $g^{-1} \g g \in BWB$ for all $g
 \in G$ and $\g \in \G$. So if $\ka_\g(gB)$ has length $n$, then $g^{-1} \g g \in B w B$
 for some $w = t_{i_1}\cdots t_{i_n} \in W$ of length $n$.
 Since $g^{-1} \g^2 g \in BwB \cdot BwB$, then $\ka_{\g^2}(gB)$ has
 length $2n$ if and only if the word $w^2$ has length $2n$, which
 is equivalent to $BwB \cdot BwB = B w^2 B$.

 \begin{proposition}\label{typeonetaillessgallery}
Let $\ka_\g(gB)$ be a type $1$ tailless closed gallery in $X_\G$. Let $w \in W$
be such that $g^{-1} \g g \in B w B$.
 Then its length $n=3m$ is a multiple of $3$
and $w \in \{(t_3t_2t_1)^m, (t_2t_1t_3)^m, (t_1t_3 t_2)^m\}$.
\end{proposition}

\begin{proof} Write $w = t_{i_1}\cdots t_{i_n}$. Since $\ka_\g(gB)$
has type $1$, $w$ is one of the three length $n$ words:
$t_3t_2t_1...$ or $t_2t_1t_3...$, or $t_1t_3 t_2...$. One checks
easily that if $n$ is not a multiple of $3$, then the length of
$w^2$ is less than $2n$, while if $n$ is a multiple of $3$, the
length of $w^2$ is $2n$.
\end{proof}

We want to count the number of type $1$ tailless closed geodesic
galleries in $X_\G$ of length $3n$. Before doing this, some remark
is in order. Note that $\sigma t_1t_3t_2 \sigma^{-1} = t_2 t_1 t_3$
and $\sigma^2 t_1t_3t_2 \sigma^{-2} = t_3 t_2 t_1$. By applying
suitable powers of $t_2 t_1 t_3$ to $gB$, $g\sigma B$ and $g\sigma^2
B$, we obtain all tailless type $1$ geodesic galleries starting at
the chamber $gC_0$. {\it In what follows,
we shall use the three $B$-coset representatives for each chamber,
but call $\ka_\g(gB)$ type $1$ tailless of length $3m$ if and only
if $g^{-1} \g g \in B(t_2t_1t_3)^mB$}.
Recall the operator $L_B = L_{t_2\sigma^2}$ defined by (\ref{LB}).
Further, $L_B$ on $X_\G$ can be interpreted as the adjacency matrix
on directed chambers $(C, e)$, where $e$ is a type $1$ edge of the
chamber $C$ in $X_\G$.

\begin{theorem}\label{numberofgalleries}
For $n \ge 1$, $\Tr L_B^{3n}$ counts the number of type $1$
tailless closed galleries in $X_\G$ of length $3n$.
\end{theorem}

\begin{proof}
Write $Bt_2t_1t_3B = \coprod_{1 \le l \le M} w_lB$ as a disjoint
union. As $B(t_2t_1t_3)^n B = (Bt_2t_1t_3B)^n$  and the length of
$(t_2t_1t_3)^n$ is $3n$, we have $B(t_2t_1t_3)^n B = \coprod_{1 \le
l_1,..., l_n \le M} w_{l_1} \cdots w_{l_n}B$. Consequently,
$\ka_\g(gB)$ is a type $1$ tailless closed gallery of length $3n$ in
$X_\G$ if and only if $ g^{-1} \g g$ lies in $w_{l_1} \cdots
w_{l_n}B$ for some $1 \le l_1,..., l_n \le M$, that is, $\g gB =
gw_{l_1} \cdots w_{l_n}B$. As we vary $\g$ and $gB$, this amounts to
counting, for each double coset $\G g B$, the number of $w_{l_1}
\cdots w_{l_n}$'s such that $\G g B = \G gw_{l_1} \cdots w_{l_n}B$,
and then total over all double cosets $\G \backslash G /B$.

On the other hand, represent $L_B^3 = L_{t_2t_1t_3}$ by a square
matrix with rows and columns parametrized by the characteristic
functions of $\G \backslash G/B = \coprod_{1 \le i \le N} \G g_i B$.
Then the $ij$ entry of $L_B^3$ is one if $\G g_j B = \G g_i w_l B$
for some $1 \le l \le M$, and zero otherwise. Therefore the trace of
the $n$th power of $L_B^3$ gives the number of type $1$ tailless
closed galleries in $X_\G$ of length $3n$.
\end{proof}

\subsection{The type $1$ gallery zeta function of $X_\G$} A type $1$ tailless closed gallery $\ka_\g(gB)$ is called {\it primitive}
if it is not a repetition of another closed gallery of shorter
length. If $\ka_\g(gB)$ is a primitive tailless type $1$ closed
gallery of length $n$, then so is the same closed gallery with a
different starting chamber. These galleries are said to be {\it
equivalent}. Denote by $[\ka_\g(gB)]$ the collection of the $n$
galleries equivalent to $\ka_\g(gB)$.

The type $1$ gallery zeta function of $X_\G$ is defined as an Euler
product:
\begin{eqnarray}\label{defineZ2}
Z_2(X_\G, u) = \prod_{\g \in [\G]} \prod_{[\ka_\g(gB)]} (1 -
u^{l(\ka_\g(gB))})^{-1}
\end{eqnarray}
where $[\ka_\g(gB)]$ runs through the equivalence classes of
primitive, tailless, type $1$ galleries in $[\g]_B$.

\begin{theorem}\label{rationalZ2} The type $1$ gallery zeta
function of $X_\G$ is a rational function, given by
\begin{eqnarray}\label{L2asdetLB}
Z_2(X_\G, u) = \frac{1}{\det (I - L_B u)}.
\end{eqnarray}
\end{theorem}

\begin{proof} We compute
\begin{eqnarray*}
u\frac{d}{du}\log Z_2(X_\G, u) &=& u\frac{d}{du}\bigg( \sum_{\g
\in [\G]} \sum_{[\ka_\g(gB)]} \sum_{m \ge
1}\frac{u^{l(\ka_\g(gB))m}}{m}\bigg) \\
&=& \sum_{\g \in [\G]} \sum_{[\ka_\g(gB)]}~ \sum_{m \ge
1}l(\ka_\g(gB)) u^{l(\ka_\g(gB))m} \\
&=& \sum_{\g \in [\G]} \sum_{\substack{\ka_\g(gB)
\text{primitive},\\\text{ tailless, type $1$}}} ~\sum_{m \ge 1}
u^{l(\ka_\g(gB))m}
\end{eqnarray*}
since there are $l(\ka_\g(gB))$ galleries in $[\g]_B$ equivalent
to $\ka_\g(gB)$. As we get all tailless type $1$ galleries by
repeating the primitive ones, the above can be rewritten as
\begin{eqnarray*}
u\frac{d}{du}\log Z_2(X_\G, u) &=& \sum_{\g \in [\G]}
\sum_{\ka_\g(gB) \text{tailless, type $1$}} u^{l(\ka_\g(gB))}\\
&=& \sum_{m \ge 1} \Tr L_B^m u^m \qquad \qquad \qquad  \text{by
Proposition
\ref{numberofgalleries}} \\
&=& \Tr((1 - L_Bu)^{-1}L_Bu) = \Tr\bigg(-u \frac{d}{du}\log(I -
L_Bu)\bigg).
\end{eqnarray*}
Therefore $\log Z_2(X_\G, u)$ differs from $-\Tr \log(1 - L_Bu)$ by
a constant. Exponentiating both functions, using Lemma 3 of
\cite{ST} and comparing the constants, we get the desired
conclusion.
\end{proof}

\begin{remark} By Proposition \ref{typeonetaillessgallery}, the lengths of the
closed galleries occurring in the gallery zeta function are
multiples of $3$, so $\det(1 - L_Bu)$ is a polynomial in $u^3$.
\end{remark}

\section{Edge zeta functions of $X_\G$}
\subsection{The type $1$ edge zeta function of $X_\G$}

The intersection of the stabilizers in $G$ of $v_1 = K$ and $v_2 =
\sigma K$ is the group $E$ consisting of elements $k \in K$ whose
third row is congruent to $(0, 0, *)$ mod $\pi$. Therefore $E$
stabilizes the type $1$ edge $E_0 : v_1 \rightarrow v_2$.
Further, $gE_0 \mapsto gE$ is a bijection between the type $1$ edges
on $\B$ and the coset space $G/E$.

We have
$$(t_2\sigma^2)^2 = \left( \begin{matrix} \pi &  &  \\
& \pi&  \\ & &\pi^2 \end{matrix}\right) = \left( \begin{matrix} 1 &  &  \\
& 1&  \\ & &\pi \end{matrix}\right) ~(\text{in} ~G)$$ and
$$E(t_2\sigma^2)^2E = E\left( \begin{matrix} 1 &  &  \\
& 1&  \\ & & \pi \end{matrix}\right)E = \coprod_{x, ~y \in
~\oo/\pi\oo}
\left( \begin{matrix} 1 &  &  \\
& 1&  \\ x\pi & y\pi &\pi \end{matrix}\right)E.$$ Let $L_E$ be the
operator which sends a function $f$ in $L^2(G/E)$ to the function
$L_Ef$ whose value at $gE$ is given by
$$ L_Ef (gE) = \sum_{g'E \subset E(t_2\sigma^2)^2E} f(gg'E) =
\sum_{x, ~y \in ~\oo/\pi\oo} f\bigg(g\left( \begin{matrix} 1 &  &  \\
& 1&  \\ x\pi & y\pi &\pi \end{matrix}\right)E\bigg).$$ Observe
that
 left multiplications by the elements $\left( \begin{matrix} 1 &  &  \\
& 1&  \\ x\pi & y\pi &\pi \end{matrix}\right)$ map the vertex
$v_1=K$ to $v_2 = \sigma K = \diag(1, 1, \pi)K$ and $v_2K$ to its
type $1$ neighbors which are not adjacent to $v_1$. In other words,
$L_E$ may be interpreted as the \lq\lq edge adjacency operator" on
the set of type $1$ edges $G/E$ of $\B$ such that the neighbors of a
type $1$ edge $v \rightarrow v'$ are the $q^2$ type $1$ edges $v'
\rightarrow v''$ with $v''$ not adjacent to $v$.

Regard $L_E$ as an operator on the type $1$ edges in $X_\G$. Then
$\Tr L_E^n$ counts the number of type $1$ tailless cycles of length
$n$ in $X_\G$. Similar to the type $1$ gallery zeta function, we
define the type $1$ edge zeta function on $X_\G$ to be
\begin{eqnarray}\label{defLZ1}
Z_1(X_\G, u) = \prod_{\g \in [\G]} \prod_{[\ka_\g(gK)]} (1 -
u^{l_A(\ka_\g(gK))})^{-1},
\end{eqnarray}
where $[\ka_\g(gK)]$ runs through the classes of equivalent
primitive tailless type $1$ cycles in $X_\G$. The same argument as
the proof of Theorem \ref{rationalZ2} shows

\begin{theorem}\label{rationalZ1}
The type $1$ edge zeta function of $X_\G$ is a rational function,
given by
\begin{eqnarray}\label{L2asdetLB}
Z_1(X_\G, u) = \frac{1}{\det (I - L_E u)}.
\end{eqnarray}
\end{theorem}

\subsection{Boundaries of tailless type $1$ closed galleries}

We characterize the boundary of a type $1$ tailless closed
gallery. Recall from Proposition \ref{typeonetaillessgallery} that
the length of such a gallery is a multiple of $3$. For $\g \in
[\G]$, let
\begin{equation}\label{geotaillessset}
\Delta_G([\g]) = \{gK \in G/K ~|~ l_G(\ka_\g(P_\g gK)) =
l_G([\g])\}.
\end{equation}
\noindent By Corollaries \ref{typeoftailless} and
\ref{rankonetailless}, tailless cycles in $[\g]$ are algebraically
tailless, thus $\Delta_G([\g]) \subseteq \Delta_A([\g])$;
furthermore, the two sets agree when $[\g]$ has type $1$.

\begin{proposition}\label{boundaryofgallery}
Let $\ka_\g(gB)$ be a type $1$ tailless closed gallery of length
$3m$ in $X_\G$ with the chamber sequence
$$ gB = g_1B \to g_2B \to \cdots \to g_{3m}B \to g_{3m+1}B = \g g_1B.$$

\begin{itemize}
\item[(1)] Suppose $3m = 3\cdot 2n$ is even. Then up to
equivalence the boundary of $\ka_\g(gB)$ consists of two tailless
type $1$ edge cycles $gE= g_1E \to g_3E \to \cdots \to g_{3m-1}E \to
g_{3m+1}E = \g g_1E = \g gE$ and $g_2E \to g_4E \to \cdots \to
g_{3m}E \to g_{3m+2}E = \g g_2E$. As vertex cycles, they are
$\ka_\g(gK)$ and $\ka_\g(g_2K)$, both of type $(3m/2, 0)$.
Consequently, $[\g]$ has type $(3m/2, 0)$.

\item[(2)] Suppose $3m =3(2n+1)$ is odd. Then up to equivalence
the boundary of $\ka_\g(gB)$
 consists of one tailless type $1$ edge cycle $gE = g_1E \to g_3E \to
\cdots \to g_{6n+1}E \to g_{6n+3}E \to \g g_2E \to \g g_4E\to \cdots
\to \g g_{6n+2}E \to \g^2g_1E = \g^2 gE$. As a vertex cycle, this is
$\ka_{\g^2}(gK)$, of type $(3m, 0)$. In this case, $[\g]$ has type
$((3m-1)/2, 1)$.
\end{itemize}

\noindent In both cases, all vertices contained in the gallery
$\ka_\g(gB)$ belong to the set $\Delta_G([\g])$. Moreover, each
chamber in $\kappa_\g(gB)$ contains a unique type $1$ edge which
starts a tailless cycle in $[\g]$.
\end{proposition}

\begin{remark} The element $\g$ in case (2) is ramified rank-one
split, in view of Theorem \ref{rankoneminlength}, (1).
\end{remark}

\begin{proof} Since the edge sequences we are
considering come from every other term of the chamber sequence,
they are obtained by right multiplications by suitable $B$-coset
representatives of $B(t_2\sigma^2)^2B = \sum_{1 \le l \le q^2}
w_lB$. If the closed gallery has even length $6n$, then there are
$w_{l_1},..., w_{l_{3n}}$ with $1 \le l_1,..., l_{3n} \le q^2$ so
that for $1 \le j \le 3n$ we have $g_{2j+1}B = g_{2j-1}w_{l_j}B$.
 As explained at the beginning of the previous
section, each $g_jE$ is a type $1$ edge of the chamber $g_jB$, and
$g_{2j+1}E = g_{2j-1}w_{l_j}E$ is adjacent to $g_{2j-1}E$.
Therefore $g_1E \to g_3E \to \cdots \to g_{6n-1}E \to g_{6n+1}E =
\g g_1E$ is a type $1$ tailless edge cycle in $X_\G$. The same
holds for $g_2E \to g_4E \to \cdots \to g_{6n}E \to \g g_2E$.

To see the type of the vertex cycles $\ka_\g(gK)$ and
$\ka_\g(g_2K)$, note that $g_1w_{l_1} \cdots w_{l_{3n}}B = \g g_1B$
implies that $g_1^{-1} \g g_1 \in w_{l_1} \cdots w_{l_{3n}}B \subset
B(t_2\sigma^2)^{6n}B = B(t_2t_1t_3)^{2n}B \subset
K(t_2t_1t_3)^{2n}K$. Similarly, we also have $g_2^{-1} \g g_2 \in
K(t_2t_1t_3)^{2n}K$. A straightforward computation gives
\begin{eqnarray}\label{t2t1t3andsquare}
 ~~~~~~~~~~~~t_2t_1t_3 = \left( \begin{matrix}  & \pi^{-1} &  \\
1 & &  \\ & &\pi \end{matrix}\right) \quad \text{and} \quad
(t_2t_1t_3)^2
= \left( \begin{matrix} 1 &  &  \\
& 1&  \\ & &\pi^3 \end{matrix}\right) \quad \text{in} ~G.
\end{eqnarray}
This shows that $\ka_\g(gK)$ and $\ka_\g(g_2K)$ both have type $(3n,
0)$. As they are tailless type $1$ cycles, we know that $[\g]$ has
the same type and the vertices on $\ka_\g(gK)$ and $\ka_\g(g_2K)$,
that is, the vertices contained in the gallery $\ka_\g(gB)$, all
belong to $\Delta_G([\g])$.

If, however, the gallery has odd length $3m=3(2n+1)$, then the
boundary sequence is $g_1E \to g_3E \to \cdots \to g_{6n+1}E \to
g_{6n+3}E \to g_{6n+5}E = \g g_{2}E \to \g g_4E \to \cdots \to \g
g_{6n+2}E \to \g g_{6n+4}E = \g^2 g_1E$. The same argument shows
that it is a tailless type $1$ edge cycle in $X_\G$, and as a vertex
cycle, it is $\ka_{\g^2}(gK)$. Further, we have $g^{-1} \g^2 g \in
K(t_2t_1t_3)^{2m}K$. Therefore $\ka_{\g^2}(gK)$ has type $(3m, 0)$
by (\ref{t2t1t3andsquare}). Since $m$ is odd, $g^{-1}\g g \in
K(t_2t_1t_3)^{m}K$ has type $(3n+1, 1)$. If $\ka_\g(gK)$ is not
tailless in $[\g]$, then $l_G([\g]) \le 3n+1$, which in turn implies
$l_G([\g^2]) \le 6n+2$, contradicting $l_G(\ka_{\g^2}(gK)) =
l_G[\g^2] = 6n+3$ since $\ka_{\g^2}(gK)$ is tailless. Thus
$\ka_\g(gK)$ is tailless so that $[\g]$ has type $(3n+1, 1)$. This
also shows that the vertices in the gallery $\ka_\g(gB)$ lie in
$\Delta_G([\g])$.

Finally, the unique type $1$ edge of each chamber which starts a
cycle in $[\g]$ is the one which shows up in the edge sequences in
(1) and (2), respectively.
\end{proof}

The proposition above says that if $[\g]_B$ contains a tailless type $1$ closed gallery, then either $[\g]$ has type $(3n, 0)$, or it has
type $(3n+1, 1)$. Further, each chamber of such a gallery has its
vertices contained in the set $\Delta_G([\g])$ with a unique type $1$ edge which starts a tailless cycle in $[\g]$. Now we show that
the last statement characterizes the chambers which start a tailless
type $1$ closed gallery in $X_\G$.

Given $[\g]$ with type as described above, let $C$ be a chamber
whose three vertices are contained in the set $\Delta_G([\g])$  with
a unique type $1$ edge $E'$ which is the starting edge of a tailless
cycle in $[\g]$. Initially, the chamber $C$ has three possible
labels: $gB$, $g\sigma B$ and $g\sigma^2 B$. The edge $E'$ then
determines the unique labeling, say, $gB$ so that $E'$ is labeled as
$gE$. The three vertices of $gB$ are $gK$, $g\sigma K$ and $g
\sigma^2 K$.
 Denote by $g \A$ the apartment containing $gB$ and $\g gB$. Up to translation
 by an element in $B$, we may assume that $\A$ is the standard apartment whose
chambers are represented by $DS_3 B$, where $D$ is the group of
diagonal matrices in $G$ and $S_3$ is the subgroup of permutation
matrices in $G$. Therefore $g^{-1} \g g = M sb$ for some $M \in D$,
$s \in S_3$ and $b \in B$. The cycles in $[\g]$ starting at the
vertices of $C$ are tailless and have the same type and length as
$[\g]$.

Case (I). $[\g]$ has type $(3n, 0)$. We have, by assumption, that
$g^{-1}\g g$, $\sigma^{-1}g^{-1}\g g \sigma$ and $\sigma g^{-1}\g g
\sigma^{-1}$ all lie in $K\diag(1,1,\pi^{3n})K$. Therefore $M =
\diag(1, 1, \pi^{3n})$ from $g^{-1}\g g \in K\diag(1,1,\pi^{3n})K$.
Writing $\sigma = \diag(1, 1, \pi) s_3$ with $s_3 \in S_3$, we
proceed to determine $s$ using the other two conditions. Since
\begin{eqnarray*}
\sigma^{-1}g^{-1}\g g \sigma &=& s_3^{-1}\diag(1, 1,
\pi^{-1})\diag(1, 1, \pi^{3n}) s \sigma b' \qquad \text{since}
~B\sigma = \sigma B \\
&=& s_3^{-1}\diag(1, 1, \pi^{-1})\diag(1, 1, \pi^{3n}) s ~\diag(1,
1, \pi) s_3 b'
\end{eqnarray*} and $s \diag(1, 1, \pi)$ is $\diag(\pi, 1, 1)s$
or $\diag(1, \pi, 1) s$ or $\diag(1, 1, \pi) s$ depending on the the
first, second, or third row of $s$ is $(0 ~0 ~1)$, in order that
$\sigma^{-1}g^{-1}\g g \sigma \in K\diag(1,1,\pi^{3n})K$, we must
have the third row of $s$ being $(0 ~0 ~1)$. Similarly, $\sigma
g^{-1}\g g \sigma^{-1} \in K\diag(1,1,\pi^{3n})K$ implies the first
row of $s$ should be $(1 ~0 ~0)$. Therefore $s$ is the identity
matrix and hence $g^{-1} \g g = \diag(1,1, \pi^{3n})b$, showing that
$\ka_\g(gB)$ is a tailless type $1$  closed gallery of length $6n$.

Case (II) $[\g]$ has type $(3n+1, 1)$. Since $\Delta_G([\g]) \subset
\Delta_G([\g^2])$ and $[\g^2]$ has type $1$, we may use the result
above to conclude that there is a labeling of $C$ by $gB$ such that
$\ka_{\g^2}(gB)$ is a tailless type $1$ gallery of length $2(6n+
3)$. In other words, $g^{-1}\g^2 g \in B(t_2 t_1 t_3)^{2(2n+1)}B$.
Since the vertices of $gB$ are in $\Delta_G([\g])$, we know $g^{-1}
\g g \in K(t_2t_1t_3)^{2n+1}K$. This condition allows us to write
$g^{-1} \g g = M sb$ with $M = (t_2t_1t_3)^{2n+1}$, $s \in S_3$ and
$b \in B$. A similar argument as in Case (I) shows that the
remaining two conditions force $s$ to be the identity matrix.
Therefore $g^{-1}\g g \in B(t_2t_1t_3)^{2n+1}B$, implying that
$\ka_\g(gB)$ is a tailless type $1$ closed gallery.

We record the above result in

\begin{proposition}\label{cycletogallery}
Suppose $[\g]$ has type $(3n, 0)$ or
 it is ramified rank-one split of type $(3n+1, 1)$.
Then for any chamber $C$ whose vertices belong to $\Delta_G([\g])$ with a unique
type $1$ edge which starts a tailless cycle in $[\g]$, there is a
unique labeling of $C$ by $gB$ such that $\kappa_\g(gB)$ is a
tailless type $1$ closed gallery of even length $6n$ if $[\g]$ has
type $(3n, 0)$, or odd length $3(2n+1)$ if $[\g]$ has type
$(3n+1,1)$.
\end{proposition}

\subsection{Comparison between type $1$ chamber zeta function and
type $2$ edge zeta function}

The type $2$ cycles are obtained from the type $1$ cycles traveled
in reverse direction, hence their algebraic length is doubled while
the geometric length remains the same. Consequently the type $2$
edge zeta function of $X_\G$ is equal to $Z_1(X_\G, u^2)$.

The following theorem compares the difference between the numbers of
type $2$ tailless edge cycles and type $1$ tailless closed
galleries.

\begin{theorem}\label{compareZ2andZ1}
\begin{eqnarray*}
& & u\frac{d}{du}\log Z_1(X_\G, u^2) - u\frac{d}{du}\log Z_2(X_\G,
-u)\\ &=& \sum_{n \ge 1}\big(\sum_{ ~[\g] ~{\rm unramified ~rank-one
~split ~of ~type} ~(3n, ~0)} 2\vol([\g]) u^{2l_A([\g])} \\&+&
\sum_{[\g] ~{\rm ramified ~rank-one ~split ~of ~type} ~(3n, ~0)}
\vol([\g])
u^{2l_A([\g])}\\
&+& \sum_{[\g] ~{\rm ramified ~rank-one ~split ~of ~type} ~(3n+1,
~1)} \vol([\g])u^{l_A([\g^2])}\big).
\end{eqnarray*}
\end{theorem}

\begin{proof} Combining Propositions \ref{boundaryofgallery} and \ref{cycletogallery}
as well as the proof of Theorem \ref{rationalZ2}, we have
\begin{eqnarray*}
& & u\frac{d}{du}\log Z_2(X_\G, -u) = \sum_{\g \in [\G]}
~\sum_{\ka_\g(gB) \text{~tailless, type $1$}} (-u)^{l(\ka_\g(gB))}\\
&=& \sum_{n \ge 1} ~\big(\sum_{\g \in [\G], ~[\g] ~{\rm
 of ~type} ~(3n, ~0)} N_B(\g) u^{6n} - \sum_{\g \in [\G],
 ~[\g] ~{\rm ramified ~rank-one ~split ~of ~type} ~(3n+1, ~1)} N_B(\g)u^{6n+3}\big),
\end{eqnarray*}
where $N_B(\g)$ is the number of chambers with vertices $P_\g gK$,
where $gK \in C_{P_\g^{-1} \G P_\g}(r_\g)\backslash \Delta_G([\g])$,
and containing a unique type $1$ edge which starts a tailless cycle
in $[\g]$. On the other hand, for type $1$ cycles we have
\begin{eqnarray*}
u\frac{d}{du}\log Z_1(X_\G, u^2) &=& \sum_{\g \in [\G]}
~\sum_{\ka_\g(gK) \text{~tailless, type $1$}} 2u^{2l_A(\ka_\g(gK))}\\
&=& \sum_{n \ge 1} ~\sum_{\g \in [\G], ~[\g] ~{\rm of ~type} ~(3n,
~0)} 2N_K(\g) u^{6n},
\end{eqnarray*}
where the number $N_K(\g)$ of tailless type $1$ cycles in $[\g]$ was
calculated in \S5 and \S6.
We shall compare this with the number $N_B(\g)$.
Recall that for $[\g]$ of type $1$, we have $\Delta_G([\g])=
\Delta_A([\g])$.

Case I. $\g$ is split with type $(3n, 0)$.  Then $r_\g = \diag(1, a,
b)$, where $1, a, b$ are distinct with $\ord_\pi$$ (a) = 0$ and
$\ord_\pi$$ b = 3n$. Put $\delta = \ord_\pi$$ (1 - a)$.  The centralizer
$C_G(r_\g)$ consists of diagonal elements in $G$. By Corollary
\ref{primitivesplit}, $C_{P_\g^{-1} \G P_\g}(r_\g) \backslash
\Delta_A([\g])$ has cardinality $N_K(\g) = \vol([\g])q^{\delta}$ and is
represented by vertices $h_{i,j}v_xK$, where $h_{i,j} = \diag(1,
\pi^i, \pi^j) \in C_{P_\g^{-1} \G P_\g}(r_\g) \backslash
C_G(r_\g)/(C_G(r_\g) \cap K)$ and
$v_x = \left(\begin{matrix} 1 & x &  \\
& 1 &  \\ & & 1\end{matrix}\right)$ with $x \in
\pi^{-\delta}\oo/\oo$.
The type $1$ tailless cycle $\kappa_\g(P_\g h_{i, j}v_xK)$ is $P_\g
h_{i, j}v_xK \rightarrow P_\g h_{i, j+1}v_xK \cdots \rightarrow P_\g
h_{i, j+3n}v_xK  = \g P_\g h_{i, j}v_xK$ by Corollary
\ref{primitivesplit}.

There are $q+1$ chambers sharing the type $1$ edge $K \rightarrow
\diag(1, 1, \pi)K$ with the third vertex being $u_c K :=
\left(\begin{matrix}\pi & c & \\ & 1 & \\ & & \pi
\end{matrix}\right)K$ with $c \in \oo/\pi\oo$ and $u_\infty K:=
\left(\begin{matrix}1 &  & \\ & \pi & \\ & & \pi
\end{matrix}\right)K$. Left multiplication by $h_{i,j}v_x$ sends the
type $1$ edge to $h_{i, j}v_x K \rightarrow h_{i, j+1}v_xK$ and the
third vertex to $h_{i, j}v_x u_c K = \left(\begin{matrix}1 & (c + x)/\pi & \\
& \pi^{i-1} & \\ & & \pi^j
\end{matrix}\right)K$ and $h_{i, j}v_x u_\infty K =
\left(\begin{matrix}1 & x\pi & \\ & \pi^{i+1} & \\ & & \pi^{j+1}
\end{matrix}\right)K$.
We count the number of such vertices belonging to $C_{P_\g^{-1} \G
P_\g}(r_\g) \backslash \Delta_A([\g])$.

There is only one integral $x$, namely, $x = 0$. When $\delta = 0$,
each type $1$ edge $h_{i, j}v_0K \rightarrow h_{i, j+1}v_0 K$ forms
a chamber with only two vertices $h_{i+1, j+1}v_0 K$ and $h_{i-1, j}
v_0 K$ in $C_{P_\g^{-1} \G P_\g}(r_\g) \backslash \Delta_A([\g])$.
Hence the number of type $1$ tailless galleries in $[\g]_B$ is
$N_B(\g) = 2\#(C_{P_\g^{-1} \G P_\g}(r_\g) \backslash
\Delta_A([\g])) = 2N_K(\g).$

Next assume $\delta \ge 1$. In this case, each type $1$ edge $h_{i,
j}v_0 K \rightarrow h_{i, j+1}v_0 K$ forms a chamber with the $q+1$
vertices $h_{i, j}v_0 u_cK$ and $h_{i, j}v_0 u_\infty K$ in
$C_{P_\g^{-1} \G P_\g}(r_\g) \backslash \Delta_A([\g])$. The same
holds when $h_{i,j}v_0$ is replaced by $h_{i,j}v_x$ for $-1 \ge
\ord_\pi$$ x \ge -\delta + 1$. This gives rise to $(q+1)(q^{\delta -
1} - 1)$ chambers. Finally, when $\ord_\pi$$ x = - \delta$, each type $1$ edge $h_{i, j}v_x K \rightarrow h_{i, j+1}v_x K$ forms a chamber
with only one vertex $h_{i, j}v_x u_\infty K$ in $C_{P_\g^{-1} \G
P_\g}(r_\g) \backslash \Delta_A([\g])$, so there are $(q-1)q^{\delta
- 1}$ chambers. Put together, we get $N_B(\g) = \vol([\g]) \big(q+1 +
(q+1)(q^{\delta - 1} - 1) + (q-1)q^{\delta - 1}\big) =
\vol([\g])2q^\delta = 2N_K(\g).$

Hence there is no contribution from $[\g]$, split type $1$,  in
$u\frac{d}{du}\log Z_1(X_\G, u^2) - u\frac{d}{du}\log Z_2(X_\G,
-u)$.

Case II. $\g$ is unramified rank-one split with type $(3n, 0)$. In
this case $r_\g = \left(
\begin{matrix} a &  &  \\ & e & dc \\ & d& e+db \end{matrix}\right)$,
and the eigenvalues $a$, $e + d \lambda$ and $e + d \bar {\lambda}$
of $\g$  generate an unramified quadratic extension $L$ over $F$.
The type assumption on $\g$ implies that $\ord_\pi$$ a = 3n$ and
$\min(\ord_\pi$$ e, \ord_\pi$$ d) = 0$ so that $e + d \lambda$ and $e +
d \bar {\lambda}$ are units in $L$. Let $\delta = \ord_\pi$$ d$.

As discussed in \S \ref{centralizers},
the double cosets $C_{P_\g^{-1} \G P_\g}(r_\g)\backslash
C_G(r_\g)/C_G(r_\g)\cap K$ are represented by $h_m = \diag(\pi^m, 1,
1)$, $m \mod \vol([\g])$. By Proposition
\ref{rankonenumberofalgtailless}, $C_{P_\g^{-1} \G P_\g}(r_\g)
\backslash \Delta_A([\g])$ has cardinality $N_K(\g) =
\vol([\g])\frac{q^\delta + q^{\delta - 1}-2}{q-1}$ and is represented
by $h_mg_{i,j,u}K$ and $h_m g_{i,z}K$, where $m \mod \vol([\g])$,
$g_{i,j,u} = \left(
\begin{matrix}
1 &  & \\
& \pi^{i-j} & u\\
& & \pi^j
\end{matrix} \right)$ with $0 \le j \le i \le \delta$, $u \in \oo^\times/\pi^{i-j}\oo$
for $j<i$ and $u = 0$ for $j=i$, and $ \quad g_{i,z} = \left(
\begin{matrix}
1 &  & \\
& \pi^{i} & z\\
& & 1
\end{matrix} \right)$ with $1 \le i \le \delta$ and $z \in \pi
\oo/\pi^i \oo$. Let $g = h_mg_{i,j,u}$ or $h_m g_{i,z}$. Then, by
Proposition \ref{rankonenumberofalgtailless}, the type $1$ tailless
closed geodesic $\kappa_\g(P_\g gK)$ is given by $P_{\g}gK
\rightarrow P_{\g}g \diag(\pi, 1, 1)K \rightarrow \cdots \rightarrow
P_{\g}g \diag(\pi^{3n}, 1, 1)K = \g P_{\g}gK$.

It remains to count the number of chambers with vertices in
$C_{P_\g^{-1} \G P_\g}(r_\g) \backslash \Delta_A([\g])$ containing a
given type $1$ edge $gK \rightarrow g\diag(\pi, 1, 1)K$ for $g
=h_mg_{i,j,u}$ or $h_mg_{i,z}$. When $\delta = 0$, there are no
$g_{i,z}$ and only one $g_{i, j, u}$, equal to the identity matrix,
hence the vertices in $C_{P_\g^{-1} \G P_\g}(r_\g) \backslash
\Delta_A([\g])$ are $h_mK$, $m \mod \vol([\g])$.  It is clear that
there are no chambers formed by these vertices. Hence $N_K(\g) =
\vol([\g])$ and $N_B(\g) = 0$ when $\delta = 0$.

Next assume $\delta \ge 1$. There are $q+1$ chambers in $\B$ sharing
the type $1$ edge $K \rightarrow \diag(\pi, 1, 1)K$ with the third
vertex
being $w_x K:= \left(\begin{matrix}\pi & & \\
& \pi & x \\ & & 1
\end{matrix}\right)K$ with $x \in \oo/\pi\oo$ and $w_\infty K:= \diag(1, \pi^{-1}, 1)K$,
respectively. Left multiplication by $g = h_mg_{i,j,u}$ or $
h_mg_{i,z}$ sends the edge $K \rightarrow \diag(\pi, 1, 1)K$ to the
type $1$ edge $gK \rightarrow g\diag(\pi, 1, 1)K$, so we need to
count the number of distinct vertices among $gw_xK$ and $gw_\infty
K$ which fall in $C_{P_\g^{-1} \G P_\g}(r_\g) \backslash
\Delta_A([\g])$. Observe that
$$h_mg_{i,j,u}w_x K = \left(\begin{matrix} \pi^{m+1} & & \\ &
\pi^{i-j+1} & x \pi^{i-j} + u \\ & & \pi^j \end{matrix}\right)K,
\qquad  h_mg_{i,j,u}w_\infty K= \left(\begin{matrix} \pi^{m} & & \\
& \pi^{i-j-1} &  u \\ & & \pi^j \end{matrix}\right)K,$$
$$h_mg_{i,z}w_x K= \left(\begin{matrix} \pi^{m+1} & & \\ &
\pi^{i+1} & x \pi^{i} + z \\ & & 1 \end{matrix}\right)K, \qquad
\text{and} \qquad h_mg_{i,z}w_\infty K =\left(\begin{matrix} \pi^{m} & & \\
& \pi^{i-1} & z \\ & & 1 \end{matrix}\right)K.$$

It is straight forward to check that, for $0 \le i \le \delta - 1$,
all $gw_x K$ and $gw_\infty K$ are distinct vertices in
$C_{P_\g^{-1} \G P_\g}(r_\g) \backslash \Delta_A([\g])$, thus there
are $\vol([\g])(q+1)\frac{q^\delta + q^{\delta - 1}-2}{q-1}$ chambers.
When $i = \delta$, for each $g$ above, only $gw_\infty K$ lies in
$C_{P_\g^{-1} \G P_\g}(r_\g) \backslash \Delta_A([\g])$, hence there
are $\vol([\g])(q^\delta + q^{\delta - 1})$ chambers. Altogether,
$N_B(\g)$ is equal to $2N_K(\g) - 2\vol([\g])$ for $\delta \ge 0$.

In conclusion, the contribution of an unramified rank-one split
$[\g]$ of type $1$ in $u\frac{d}{du}\log Z_1(X_\G, u^2)
-u\frac{d}{du}\log Z_2(X_\G, -u)$ is $2\vol([\g])u^{2l_A([\g])}$.

Case III. $\g$ is ramified rank-one split with type $(3n, 0)$. Then
$r_\g = \left(
\begin{matrix} a &  &  \\ & e & dc \\ & d& e+db \end{matrix}\right)$
 and the eigenvalues $a$, $e + d \lambda$ and $e + d \bar {\lambda}$
of $\g$  generate a ramified quadratic extension $L$ over $F$. In
this case, $\ord_\pi$$ a = 3n$ and $\ord_\pi$$ e = 0$ so that $e + d
\lambda$ and $e + d \bar {\lambda}$ are units in $L$. Let $\delta =
\ord_\pi$$ d$.

As discussed in \S \ref{centralizers}, $C_{P_\g^{-1} \G P_\g}(r_\g) \backslash
C_G(r_\g)/C_G(r_\g)\cap K$ has cardinality $\vol([\g])$, and it is
represented by $h= \diag(\pi^m, 1, 1)$ with $0 \le m \le
(\vol([\g])-1)/2$ and $ \diag(\pi^m, 1, 1)\pi_L$ with $0 \le m \le
(\vol([\g])-3)/2$ if $\vol([\g])$ is odd, and by $h = \diag(\pi^m, 1, 1)$
and $\diag(\pi^m, 1, 1)\pi_L$ with $m \mod \vol([\g])/2$ if $\vol([\g])$
is even. Here
$\pi_L = \left(\begin{matrix} 1 & & \\ & & c\\
 & 1 & b \end{matrix} \right)$ is imbedded in $G$.

It follows from
Proposition \ref{rankonenumberofalgtailless} that $C_{P_\g^{-1} \G
P_\g}(r_\g) \backslash \Delta_A(\g)$ is represented by $hg_{i,j,u}K$
for $g_{i,j,u}$ as in Case II and $h$ as above, so the total number
of vertices is $\vol([\g])(q^{\delta+1}-1)/(q-1)= N_K(\g)$.
Now, for any $gK = hg_{i,j,u}K$ in $\Delta_A([\g])$, the type $1$
tailless cycle $\kappa_\g(P_\g gK)$ is $P_\g gK \rightarrow P_\g
g\diag(\pi, 1, 1)K \rightarrow \cdots \rightarrow
 P_\g g\diag(\pi^{3n}, 1, 1)K = \g P_\g gK$ by Proposition \ref{rankonenumberofalgtailless}.

 To count the number of chambers we proceed as in Case II by
 counting, for each $g = hg_{i,j,u}$, the number of $gw_x K$ and
 $gw_\infty K$ which lie in $C_{P_\g^{-1} \G P_\g}(r_\g) \backslash
 \Delta_A(\g)$.

 We first discuss the case $\delta = 0$. Then there is only
 one $g_{0,0,u}$, equal to the identity matrix. All representatives
 are given by $hK$.
 Observe that $\diag(\pi^m, 1, 1)\pi_L K =
 \left(\begin{matrix}\pi^m & & \\ & \pi & 0 \\ & & 1
 \end{matrix}\right) K$. So there is only one vertex
 $gw_0 K$ which will form a chamber containing the
 type $1$ edge $gK \rightarrow g\diag(\pi, 1, 1)K$. Hence the
  number of chambers is $N_B(\g)= \vol([\g]) = 2N_K(\g) -
 \vol([\g])$ for $\delta = 0$.

 Now assume $\delta \ge 1$.  One sees from the explicit computation
 in Case II that for $g = hg_{i,j,u}$,
 all $q+1$ vertices $gw_x K$ and
 $gw_\infty K$ are distinct vertices in  $C_{P_\g^{-1} \G P_\g}(r_\g) \backslash \Delta_A([\g])$
 provided that $0 \le i \le \delta - 1$; when $i = \delta$, only one
 vertex, $gw_\infty K$, lies in $C_{P_\g^{-1} \G P_\g}(r_\g) \backslash \Delta_A(\g)$.
 This gives $\vol([\g])\big((q^\delta - 1)(q+1)/(q-1) + q^\delta \big) = \vol([\g])\big(2(q^{\delta + 1}
 - 1)/(q-1) - 1\big)$ chambers. Therefore $N_B(\g) = 2N_K(\g) - \vol([\g])$ for $\delta \ge 1$.

 This shows that the contribution of a ramified rank-one split
$[\g]$ of type $1$ in $u\frac{d}{du}\log Z_1(X_\G, u^2)
-u\frac{d}{du}\log Z_2(X_\G, -u)$ is $\vol([\g])u^{2l_A([\g])}$.

Finally we consider $[\g]$ of type $(3n+1, ~1)$. This happens only
when $\g$ is ramified rank-one split with eigenvalues $a, e+
d\lambda, e + d \bar{\lambda}$, where $a, e, d \in F$, $\ord_\pi$$ a =
3n+2$, $\ord_\pi$$ e \ge 1$ and $\delta = \ord_\pi$$ d = 0$ by the
analysis above Theorem \ref{rankoneminlength}. As noted before, such
$[\g]$ has no contribution to $L_1(X_\G, u^2)$ and the length of a
type $1$ tailless gallery in $[\g]_B$ is $6n+3$. Its contribution in
$u\frac{d}{du}\log Z_2(X_\G, -u)$ is $-N_B(\g) u^{6n+3}$ with
$N_B(\g) = \#C_{P_\g^{-1} \G P_\g}(r_\g) \backslash \Delta_G([\g])$.
Since $\delta = 0$ and $\mu = 0$ by the remark following Theorem
\ref{numberinrankoneclass}, we have $\Delta_G([\g]) =
\Delta_A([\g])$ such that $N_B(\g) = \vol([\g])$ by Corollary
\ref{rankonedoublecosets}.

This completes the proof of the theorem.
\end{proof}

\section{The proof of the Main Theorem}

\subsection{Type $1$ zeta function}

As defined in (\ref{defLZ1}), the type $1$ edge zeta function of the
quotient $X_\G$ is
\begin{equation}\label{zeta}
\begin{aligned}
 Z_{1}(X_\G, u)
= \prod_{\g \in [\G], ~[\g] ~\text{type $1$}}~\prod_{\substack{ \kappa_\g(gK) \in [\g] ~\text{primitive, tailless}\\
\text{up to equivalence}}} (1 - u^{l_A(\kappa_\g(gK))})^{-1}.
\end{aligned}
\end{equation}
Note that $l_A(\kappa_\g(gK))= l_G(\kappa_\g(gK)) = l_A([\g]) =
l_G([\g])$ is the length of $[\g]$. We proceed to investigate its
logarithmic derivative.

Although the zeta function only concerns type $1$ tailless cycles,
to describe it we shall involve all homotopy cycles. First we
introduce the numbers $P_{n,m}$, $Q_{n,m}$, and $R_{n,m}$ which
count the algebraically tailless homotopy cycles of type $(n, m)$
arising from split, unramified rank-one split, and ramified rank-one
split $\g$'s, respectively:
\begin{equation}\label{pnm}
\begin{aligned}
P_{n,m} = ~ \sum_{\substack{\g \in [\G] ~\text{split}\\ ~[\g]~
\text{of type }(n,m) }} ~\#(C_{P_\g^{-1} \G P_\g}(r_\g)\backslash
\Delta_A([\g])),
\end{aligned}
\end{equation}
\begin{equation}\label{qnm}
\begin{aligned}
Q_{n,m} = ~ \sum_{\substack{\g \in [\G] ~\text{unram. rank-one split}\\
~[\g]~ \text{of type }(n,m) }} ~\#(C_{P_\g^{-1} \G
P_\g}(r_\g)\backslash \Delta_A([\g])),
\end{aligned}
\end{equation}
\begin{equation}\label{rnm}
\begin{aligned}
R_{n,m} = ~ \sum_{\substack{\g \in [\G] ~\text{ram. rank-one split}\\
~[\g]~ \text{of type }(n,m) }} ~\#(C_{P_\g^{-1} \G
P_\g}(r_\g)\backslash \Delta_A([\g])).
\end{aligned}
\end{equation}

The following expression describes the type $1$ edge zeta function
in terms of the number of tailless type $1$ homotopy cycles in
$X_\G$.

\begin{proposition}\label{zetaassum}
$$ u \frac{d}{du} \log Z_1(X_\G, u) =\sum_{n>0}(P_{n,0} + Q_{n,0} + R_{n,0})u^n.$$
\end{proposition}
\begin{proof} By definition,
$$ \log Z_1(X_\G, u) = \sum_{\g \in [\G], ~[\g] ~\text{type $1$}}~\sum_{\substack{ \kappa_\g(gK) ~\text{primitive, tailless}\\
\text{up to equivalence}}} ~\sum_{m \ge 1}
\frac{u^{ml_A(\kappa_\g(gK))}}{m}$$ so that
\begin{eqnarray*}
u \frac{d}{du} \log Z_1(X_\G, u) &=& \sum_{\g \in [\G],
 ~[\g] ~\text{type $1$}}~\sum_{\substack{ \kappa_\g(gK) ~\text{primitive, tailless}\\
\text{up to equivalence}}} ~\sum_{m \ge 1}
l_A(\kappa_\g(gK))~u^{m l_A(\kappa_\g(gK))}\\
&=& \sum_{\g \in [\G], ~[\g] ~\text{type $1$}}~\sum_{\kappa_\g(gK) ~\text{primitive, tailless}} ~\sum_{m \ge
1}~u^{m l_A(\kappa_\g(gK))}
\end{eqnarray*}
since, as discussed in \S \ref{primitivecycles}, there are $l_A([\g])$ type $1$
tailless homotopy cycles equivalent to a given primitive tailless
type $1$ homotopy cycle in $[\g]$. Observe that the $\kappa_\g(gK)$
above runs through all primitive tailless type $1$ homotopy cycles
on $X_\G$, hence their repetitions give all type $1$ tailless
homotopy cycles. The proposition follows by noting that when a cycle
is repeated $m$ times, the length is multiplied by $m$.
\end{proof}

\subsection{The number of homotopy cycles of type $(n,m)$ }

In order to gain information on $P_{n,0}$, $Q_{n,0}$ and $R_{n,0}$,
we extend the summation to include homotopy cycles of type $(n,m)$.
Recall that the number of such cycles is $\Tr(B_{n,m})$, and cycles
with tails are also included.
Their
relation with the number of algebraically tailless cycles is given below.

\begin{proposition}\label{traceBnm}
With the same notation as in Theorem \ref{numberinrankoneclass}, we
have
\begin{eqnarray*}
& &\sum_{\substack{n,m \geq 0 \\ (n,m)\neq (0,0)}}
\Tr(B_{n,m})u^{n+2m} = \bigg(\sum_{\substack{n,m \geq 0 \\
(n,m)\neq (0,0)}}
P_{n,m}u^{n+2m}\bigg)\frac{1-u^3}{1-q^3 u^3}\\
&+& \sum_{\substack{\g \in [\G]\\ [\g] ~\text{unram. rank-one
split}}}\vol([\g])u^{l_A([\g])}\bigg(\frac{q^{\de([\g])+1}+q^{\de([\g])}-2}{q-1}+
\frac{(q+1)q^{\de([\g])+2}u^3}{1-q^3 u^3}\bigg)
\bigg(\frac{1-u^3}{1-q^2 u^3}\bigg) \\
 &+& \sum_{\substack{\g \in [\G]\\ [\g] ~\text{ram. rank-one
split}}}\vol([\g])q^{\mu([\g])}u^{l_A([\g])}\bigg(\frac{q^{\de([\g])
+1} -1}{q-1} + \frac{q^{\de([\g]) +3}u^3}{1-q^3u^3}
\bigg)\frac{1-u^3}{1-q^2u^3}.
\end{eqnarray*}
\end{proposition}
\begin{proof} Break the right side of (\ref{sumBnm}) into three parts, over
split, unramified rank-one split, and ramified rank-one split
$\g$'s, respectively. Applying Theorem \ref{numberinaclass} to the
split part and Theorem \ref{numberinrankoneclass} to the unramified
and ramified rank-one split parts, and using the definition of
$P_{n,m}$, we get the desired formula.
\end{proof}

Next we compute the number of type $1$ homotopy cycles on $X_\G$.

\begin{proposition}\label{alltypezerocycles}
With the same notation as in Theorem \ref{numberinrankoneclass}, we
have
\begin{eqnarray*}
& &\sum_{n>0} \Tr(B_{n,0})u^{n} = (1-q^{-1})\bigg(\sum_{(n, m)
\ne (0,0)} P_{n,m}u^{n+2m}\bigg)\frac{1-q^2u^3}{1-q^3u^3}\\
&+& \sum_{\substack{\g \in [\G]\\ [\g] ~\text{unram. rank-one
split}}}\vol([\g])u^{l_A([\g])}\bigg(q^{\de([\g])} + q^{\de([\g]) - 1}+
\frac{(q^2-1)q^{\de([\g]) +1}u^3}{1-q^3u^3} \bigg) \\
&+& \sum_{\substack{\g \in [\G]\\ [\g] ~\text{ram. rank-one
split}}}\vol([\g])u^{l_A([\g])}\bigg(q^{\de}(q^{\mu} - \mu) +
\frac{(q-1)q^{\de + \mu +2}u^3}{1-q^3u^3}\bigg)\\
&+& q^{-1}\sum_{n>0} (P_{n,0} + Q_{n,0} + R_{n,0})u^{n} -
2q^{-1}\sum_{\substack{\g \in [\G], ~\text{type $1$}\\ [\g]
~\text{unram. rank-one split}}}\vol([\g])u^{l_A([\g])} \\
&+& \sum_{\substack{\g \in [\G], ~\text{type $1$}\\ [\g] ~\text{ram.
rank-one split}}}\vol([\g])u^{l_A([\g])}(-q^{\mu([\g])-1} +
\mu([\g])q^{\de([\g])}).
\end{eqnarray*}
\end{proposition}
\begin{proof} By definition,
\begin{eqnarray*}
\sum_{n>0} \Tr(B_{n,0})u^{n} = \sum_{\g \in [\Gamma], ~\g \ne id}
~\sum_{\kappa_\g(gK) \in [\g] ~ \text{type $1$}}
u^{l_A(\kappa_\g(gK))}.
 \end{eqnarray*}
 We split the sum over $\g$ into three parts according to $\g$ split, unramified rank-one split,
  or ramified rank-one split. For the split part, we add (i) and (ii) of Theorem
 \ref{type0cyclesinaclass} and use the definition of $P_{n,m}$ to arrive
 at the sum $$(1-q^{-1})\bigg(\sum_{(n, m) \ne
(0,0)} P_{n,m}u^{n+2m}\bigg)\frac{1-q^2u^3}{1-q^3u^3} ~+~
q^{-1}\bigg(\sum_{n>0} P_{n,0}u^{n}\bigg).$$ For the unramified
(resp. ramified) rank-one split part, we add (A2) and (A3) (resp.
(B2) and (B3)) of Theorem \ref{numberinrankoneclass} to get
\begin{equation}\label{rankonepart}
\begin{aligned}
& &\sum_{\substack{\g \in [\G]\\ [\g] ~\text{unram. rank-one
split}}}\vol([\g])u^{l_A([\g])}\bigg(q^{\de([\g])} + q^{\de([\g]) - 1}+
\frac{(q^2-1)q^{\de([\g]) +1}u^3}{1-q^3u^3} \bigg) \\
&+& \sum_{\substack{\g \in [\G], ~\text{type $1$}\\ [\g]
~\text{unram. rank-one
split}}}\vol([\g])u^{l_A([\g])}\frac{q^{\de([\g])} +
q^{\de([\g])-1}-2}{q-1} \\
&+&\sum_{\substack{\g \in [\G]\\ [\g] ~\text{ram. rank-one
split}}}\vol([\g])u^{l_A([\g])}\bigg(q^{\de([\g])}(q^{\mu([\g])} -
\mu([\g])) +
\frac{(q-1)q^{\de([\g]) + \mu([\g]) +2}u^3}{1-q^3u^3}\bigg) \\
&+& \sum_{\substack{\g \in [\G], ~\text{type $1$}\\ [\g] ~\text{ram.
rank-one
split}}}\vol([\g])u^{l_A([\g])}\bigg(q^{\mu([\g])}\frac{q^{\de([\g])}-1}{q-1}
+ \mu([\g])q^{\de([\g])}\bigg).
\end{aligned}
\end{equation}
It follows from Proposition \ref{rankonenumberofalgtailless} and the
definitions of $Q_{n,0}$ and $R_{n,0}$ that
\begin{equation}\label{unramodds}
\begin{aligned}
& &\sum_{\substack{\g \in [\G], \text{type $1$}\\ [\g] ~\text{unram.
rank-one split}}}\vol([\g])u^{l_A([\g])}\frac{q^{\de([\g])} +
q^{\de([\g])-1}-2}{q-1} \\
&=& q^{-1}\sum_{n> 0} Q_{n,0}u^n - 2q^{-1}\sum_{\substack{\g \in
[\G], ~\text{type $1$}\\ [\g] ~\text{unram. rank-one
split}}}\vol([\g])u^{l_A([\g])}
\end{aligned}
\end{equation} and
\begin{equation}\label{ramodds}
\begin{aligned}
& &\sum_{\substack{\g \in [\G], ~\text{type $1$}\\ [\g] ~\text{ram.
rank-one
split}}}\vol([\g])u^{l_A([\g])}\bigg(q^{\mu([\g])}\frac{q^{\de([\g])}-1}{q-1}
+ \mu([\g])q^{\de([\g])}\bigg) \\
&=& q^{-1}\sum_{n> 0} R_{n,0}u^n + \sum_{\substack{\g \in [\G],
~\text{type $1$}\\ [\g] ~\text{ram. rank-one
split}}}\vol([\g])u^{l_A([\g])}(-q^{\mu([\g])-1} +
\mu([\g])q^{\de([\g])}).
\end{aligned}
\end{equation}
Finally, plug (\ref{unramodds}) and (\ref{ramodds}) into
(\ref{rankonepart}) to complete the proof.
\end{proof}

\subsection{Proof of the Main Theorem}

Combining Propositions \ref{alltypezerocycles} and \ref{traceBnm},
we obtain
\begin{eqnarray*}
 & &q \bigg( \sum_{n>0}
\Tr(B_{n,0})u^{n}\bigg) - (q-1)\bigg(\sum_{\substack{n,m \geq 0 \\
(n,m)\neq (0,0)}} \Tr(B_{n,m})u^{n+2m}\bigg)\bigg(\frac{1-q^2
u^3}{1-u^3}\bigg)\\
&=& \sum_{n>0} (P_{n,0} + Q_{n,0} + R_{n,0})u^{n} +
\sum_{\substack{\g \in [\G], ~\text{not type $1$}\\ [\g]
~\text{unram. rank-one split}}}2\vol([\g])u^{l_A([\g])} \\
&+& \sum_{\substack{\g \in [\G], ~\text{not type $1$}\\ [\g]
~\text{ram. rank-one split}}}\vol([\g])u^{l_A([\g])}(q^{\mu([\g])} -
\mu([\g])q^{\de([\g])+1}).
\end{eqnarray*}

As before, to a rank-one split $\g$, we associate $r_\g = \left(
\begin{matrix} a &  &  \\ & e & dc \\ & d& e+db \end{matrix}\right) $.
First assume $\g$ is unramified rank-one split. By Theorem
\ref{rankoneminlength}, $[\g]$ has type $(n, m) = (\ord_\pi$$ a,
\min(\ord_\pi$$ e, \ord_\pi$$ d))$, hence $[\g]$ is not of type $1$ if
and only if $a$ is a unit, which is equivalent to its inverse
$[\g^{-1}]$ having type $(m, 0)$. Note that $l_A([\g]) = 2m =
2l_A([\g^{-1}])$ by Theorem \ref{rankoneminlength}. Next assume that
$[\g]$ is ramified rank-one split. Since $\mu([\g]) = 1$ implies
$\de([\g]) = 0$, we have $q^{\mu([\g])} - \mu([\g])q^{\de([\g])+1} =
0$ in this case. Thus we need only consider the case $\mu([\g]) = 0$
so that $q^{\mu([\g])} - \mu([\g])q^{\de([\g])+1} = 1$. Then $[\g]$
is not of type $1$ if and only if $a$ is a unit, in which case it
has type $(0, \ord_{\pi}$$ e)$ if $\ord_{\pi}$$ e \le \ord_{\pi}$$ d$, and
type $(1, \ord_{\pi}$$ d)$ if $\ord_{\pi}$$ d < \ord_{\pi}$$ e$ by Theorem
\ref{rankoneminlength}. Further, we see that $[\g^{-1}]$ has type
$(\ord_{\pi}$$ e, 0)$ so that $l_A([\g]) = 2 l_A([\g^{-1}]) =
2\ord_{\pi}$$ e$ in the former case, and in the latter case,
$[\g^{-1}]$ has type $(\ord_{\pi}$$ d, 1)$, $[\g^{-2}]$ has type
$(2\ord_{\pi}$$ d + 1, 0)$ and $l_A([\g]) = 1 + 2\ord_{\pi}$$ d =
l_A([\g^{-2}])$. As remarked in \S \ref{volume}, $\vol([\g]) = \vol([\g^{-1}]) =
\vol([\g^{-2}])$ for $\g$ rank-one split. Consequently, we may replace
$\g$ by $\g^{-1}$ and rewrite
\begin{eqnarray*}
&~&\sum_{\substack{\g \in [\G], ~\text{not type $1$}\\ [\g]
~\text{unram. rank-one split}}}2\vol([\g])u^{l_A([\g])} +
\sum_{\substack{\g \in [\G], ~\text{not type $1$}\\ [\g] ~\text{ram.
rank-one split}}}\vol([\g])u^{l_A([\g])}(q^{\mu([\g])} -
\mu([\g])q^{\de([\g])+1})\\&=& \sum_{\g \in [\G], ~\text{ type $1$
unram. rank-one split}}2\vol([\g])u^{2l_A([\g])} + \sum_{\g \in [\G],
~\text{ type $1$ ram. rank-one split}}\vol([\g])u^{2l_A([\g])}
\\&+& \sum_{\g \in [\G], [\g] ~\text{of type} ~(m, 1),~\text{ram. ~rank-one ~split}}\vol([\g])u^{l_A([\g^2])},
\end{eqnarray*}
which can be expressed as the difference of the logarithmic
derivatives of $Z_1(X_\G, u^2)$ and $Z_2(X_\G, -u)$ by Theorem
\ref{compareZ2andZ1}.

 Together with Propositions \ref{lefthand1} and \ref{zetaassum}, this proves

\begin{proposition}\label{extraterms}
\begin{eqnarray*}\label{Pn0intrace}
& &u\frac{d}{du}\log \bigg(\frac{(1-u^3)^{\chi(X_\G)}}{ \det(I- A_1 u+q A_2 u^2 - q^3I
u^3)}\bigg)  \\
 &= &q \bigg( \sum_{n>0}
\Tr(B_{n,0})u^{n}\bigg) - (q-1)\bigg(\sum_{\substack{n,m \geq 0 \\
(n,m)\neq (0,0)}} \Tr(B_{n,m})u^{n+2m}\bigg)\bigg(\frac{1-q^2
u^3}{1-u^3}\bigg)\\
&=& u\frac{d}{du} \log Z_1(X_\G, u) + u\frac{d}{du} \log Z_1(X_\G,
u^2) - u\frac{d}{du} \log Z_2(X_\G, -u). \qquad \qquad
\end{eqnarray*}
\end{proposition}
\medskip

Consequently, we have
 $$\frac{(1-u^3)^{\chi(X_\G)}}{ \det(I- A_1 u+q A_2 u^2 - q^3I
u^3)} = c \frac{Z_1(X_\G, u)Z_1(X_\G, u^2)}{Z_2(X_\G, -u)} = c
\frac{\det(1 + L_Bu)}{\det(I - L_Eu) \det(I - L_E u^2)}$$ for some
constant $c$. Here the last equality comes from Theorems
\ref{rationalZ2} and \ref{rationalZ1}. Since both sides are formal
power series with constant term $1$, we find $c=1$. This concludes
the proof of the Main Theorem.

\section{Another interpretation of the zeta identity}

\subsection{Algebraic lengths and canonical algebraic length}
Let $\G$ be  a discrete cocompact torsion-free subgroup of $\PGL_n(F)$. 
An element $\gamma \in \G$ is called {\it primitive} if it is a generator of its centralizer in $\G$.
Denote the conjugacy class of $\gamma$ in $\G$ by $\langle \gamma \rangle_\G$. Call the conjugacy class $\langle \gamma \rangle_\G$
primitive if $\g$ is. Represent elements in $\PGL_n(F)$ by minimally integral matrices, i.e., matrices in $M_n(\oo) \smallsetminus \pi M_n(\oo)$; using them we define algebraic lengths of $\gamma$  and $\langle \gamma \rangle_\G$ by
$$l_A(\gamma) =  \ord_\pi(\det( \gamma)) \quad \mbox{and} \quad l_A(\langle \gamma \rangle_\G)=\min_{g \in \langle \gamma \rangle_\G} l_A(g),$$
respectively.  Extend the definition of algebraic length to $g \in \PGL_n(L)$ for any finite extension $L$ over $F$ by
$$ l_A(g) =  \frac{1}{[L:F]} \ord_\pi( N_{L/F}\circ\det(g)), $$
where $\det(g)$ is computed using minimally integral matrix representation in $\PGL_n(L)$. Note that $l_A(g)$ is independent of the choice of the field $L$ containing entries of $g$. Analogous to canonical heights on abelian varieties, define the {\it canonical algebraic length} of $\gamma$ to be
$$ L_A(\gamma) = \lim_{n \to \infty} \frac{1}{n} l_A(\gamma^n).$$
The canonical algebraic length of $\langle \gamma \rangle_\G$, denoted $L_A(\langle \gamma \rangle_\G)$, is defined similarly.

We exhibit some properties of the canonical algebraic length for $\PGL_3(F)$.

\begin{proposition}\label{canonicallength}  The following statements hold for $\g \in \G \subset \PGL_3(F)$:

1. $L_A(\gamma)= \ord_\pi$$ abc = l_A([\g])$, where $\diag(a, b, c)$ is a minimally integral matrix conjugate to $\g$.

2. $L_A(\gamma)$ is invariant under conjugation in $\PGL_3(F)$.

3. $L_A(\gamma^n)=n L_A(\gamma)$ for integers $n \ge 1$.
\end{proposition}
\begin{proof} We know that $\g$ is diagonalizable.
Let $L = F\langle \g \rangle$ be the field generated by the eigenvalues of $\g$ over $F$. Write
 $\gamma= h g h^{-1}$ for some $g=\diag(a,b,c)$ and $h$ in $\GL_3(L)$.
We may assume that $g$ and $\pi^i\gamma$ are minimally integral for some $i \in \mathbb Z$. As the characteristic polynomial of $\pi^i \gamma$ has integral coefficients, the eigenvalues $\pi^i a$, $\pi^i b$, $\pi^i c$ of $\pi^i \g$ are all integral. Since $g$ is minimally integral, we conclude that $i \geq 0$, or equivalently, $l_A(g)\leq l_A(\gamma)$. By the same argument, we see that $l_A(g^n) \leq l_A(\gamma^n)$ for all $n>0$. On the other hand, $l_A(\gamma^n)=l_A(h g^n h^{-1})\leq l_A(h)+l_A(g^n)+l_A(h^{-1})$. Consequently,
$$ \frac{1}{n} l_A(g^n) \leq \frac{1}{n}  l_A(\gamma^n) \le \frac{1}{n}  (l_A(h)+l_A(g^n)+l_A(h^{-1}))$$
for all $n \ge 1$. This shows that $\lim_{n\to \infty} \frac{1}{n} l_A(\g^n)$ exists and is equal to $\ord_\pi$$(abc)= L_A(g)= l_A([\g])$. The last equality follows from Theorem \ref{rankoneminlength}, (2).
Since $L_A(\gamma)$ is determined by its eigenvalues, it is invariant under conjugation. Further,
$L_A(\gamma^n)=\ord_\pi$$(a^nb^nc^n)=n L_A(\gamma)$.
\end{proof}

\subsection{Ihara (group) zata functions}
 When the ambient group is $\PGL_2(F)$, in \cite{Ih} Ihara defined the zeta function, using primitive conjugacy classes in $\G$,  as
$$ Z(\G,u) = \prod_{\langle \gamma \rangle_\G} (1-u^{l_A(\langle \gamma \rangle_\G)})^{-1},$$
where $\langle \gamma \rangle_\G$ runs through primitive conjugacy classes in $\G$.

Recall that an element is primitive in $\Gamma \subset \PGL_2(F)$ if and only if it is not an $m$th power of some element in $\Gamma$ with $m \ge 2$. This is not true for $\PGL_3(F)$. Instead, we have
\begin{lemma} \label{primitiveandrankonesplit}
A regular element $\gamma \in \G \subset \PGL_3(F)$ is primitive  if and only if $\gamma$ is rank-one split
and not an $m$th power of some element in $\Gamma$ with $m \ge 2$.
\end{lemma}
\begin{proof}
By Proposition \ref{centralizer},  $C_\G(\g) \cong \Z^2$  if $\gamma$ is split and
 $C_\G(\g) \cong \Z$  if $\gamma$ is rank-one split. Therefore, $\gamma$ is a generator of $C_\G(\g)$ if and only if $\gamma$ is rank-one split and not an $m$th power of some element in $\Gamma$ with $m>1$.
\end{proof}

Note that in $\PGL_2(F)$ the canonical algebraic lengths of primitive $g$ and $g^{-1}$ are the same; this no longer holds in $\PGL_3(F)$. An element $g \in \PGL_3(F)$ is said to have {\it positive} type if $L_A(g) \leq L_A(g^{-1})$, and {\it negative} type otherwise.
Given a regular discrete cocompact torsion-free subgroup $\G$ of $\PGL_3(F)$, define the positive/negative group zeta function of $\Gamma$ by
$$ Z_{\pm}(\G,u) = \prod_{\langle \gamma \rangle_\G} (1-u^{L_A(\langle \gamma \rangle_\G)})^{-1},$$
where $\langle \gamma \rangle_\G$ runs through all conjugacy classes of primitive elements in $\G$ of positive/negative type. We define the Ihara group zeta function of $\G$ by combining them together:
$$ Z(\G,u) = Z_+(\G, u) Z_-(\G, u) = \prod_{\langle \gamma \rangle_\G} (1-u^{L_A(\langle \gamma \rangle_\G)})^{-1} = \prod_{\langle \gamma \rangle_\G} (1-u^{L_A(\gamma)})^{-1},$$
where $\langle \gamma \rangle_\G$ runs through all primitive conjugacy classes of $\G$ and the last equality follows from
Proposition \ref{canonicallength}, (2). Recall from Lemma \ref{primitiveandrankonesplit} that such $\g$'s are rank-one split.

Let $\g$ be a (rank-one) primitive element in $\G$ of positive type and with rational form $r_\g$. Let
$a, b, c$ be eigenvalues of $r_\g$ up to a constant multiple. If $\g$ is unramified, we may assume
$a \in F^\times$ and $b$ and $c$ are units in an unramified quadratic extension of $F$.
We see from \S\ref{centralizers} that $C_{P_\g^{-1}\G P_\g}(r_\g)K= \langle r_\g \rangle K = \langle \diag(a, 1, 1)\rangle K = \langle \diag(\pi^{\vol([\gamma])},1,1)\rangle K$. This implies that $\ord_\pi$$ a = \vol([\g])$ since $\g$ has positive type. Thus $\diag(a, b, c)$ is a minimally integral matrix representing $\g$ and by Proposition \ref{canonicallength} we have $L_A(\g) = L_A(r_\g) = \ord_\pi$$ abc = \vol([\g])$.
Further, $L_A(\g^{-1})=2L_A(\g)=2\vol([\g])$, so $\g^{-1}$ has negative type.

Next assume that $\g$ is ramified. We distinguish two cases as in \S\ref{centralizers}. In case (i)
 where $b$ and $c$ are units in a ramified quadratic extension of $F$, we have $C_{P_\g^{-1}\G P_\g}(r_\g)K= \langle r_\g \rangle K=
 \langle \diag(a, 1, 1) \rangle K = \langle \diag(\pi^{\vol([\gamma])/2},1,1) \rangle K$. By the same argument as unramified case, we get $L_A(\g)=\vol([\gamma])/2$. For case (ii) where $b$ and $c$ are uniformizers in a ramified quadratic extension of $F$, up to a constant multiple, eigenvalues of $r_\g^2$ are $a^2/\pi, b', c'$, where $b'$ and $c'$ are units. We have $\langle r_{\g}^2 \rangle K= \langle \diag(\pi^{\vol([\g])},1,1)\rangle K = \langle \diag(a^2/\pi, 1,1)\rangle K$. Again, since $\g$ has positive type, we get $\ord_\pi$$ a^2 = 1 + \vol([\g])$ and $L_A(r_\g^2) = \ord_\pi$$ (a^2 b'c'/\pi)= \vol([\g])$. By Proposition \ref{canonicallength}, we conclude $L_A(\g)=\frac{1}{2}L_A(r_\g^2) =  \vol([\g])/2$.
 For both cases $L_A(\g^{-1})=2L_A(\g)=\vol([\g])$ so that $\g^{-1}$ has negative type.

We have shown that, a primitive $\g$ in $\G$ has positive type if and only if $\g^{-1}$ has negative type
with $L_A(\g^{-1})=2L_A(\g)$. This proves
\begin{proposition}
$ Z(\G,u)=Z_{+}(\G,u)Z_{-}(\G,u)= Z_{+}(\G,u)Z_{+}(\G,u^2)$.
\end{proposition}

\subsection{Another interpretation of the zeta identity}
Theorem \ref{compareZ2andZ1} can be rewritten as
\begin{eqnarray*}
& & u\frac{d}{du}\log Z_1(X_\G, u^2) - u\frac{d}{du}\log Z_2(X_\G,
-u)\\
&=&\sum_{ ~[\g] ~{\rm  primitive,~positive~type} } ~\sum_{n \ge 1} 2L_A(\g) u^{2nL_A(\g)}=\sum_{ ~[\g] ~{\rm primitive,~negative~type} } ~\sum_{n \ge 1} L_A(\g) u^{nL_A(\g)}\\
&=& u\frac{d}{du}\log\left( \prod_{ ~[\g] ~{\rm primitive,~negative~type} } (1- u^{L_A(\g)})^{-1} \right)=  u\frac{d}{du}\log Z_-(\G,u).
\end{eqnarray*}
After comparing the constant terms, we conclude $Z_-(\G,u)= Z_1(X_\G,u^2)/Z_2(X_\G,-u)$. Combined with Theorems \ref{rationalZ1} and \ref{rationalZ2}, we obtain
\begin{theorem} $Z_-(\G,u)$ and $Z(\G,u)$ are rational functions with the following closed forms:
$$ Z_-(\G,u)=  \frac{\det(1+L_Bu)}{\det(1-L_Eu^2)}\quad \mbox{and} \quad Z(\G,u)= \frac{\det(1+L_Bu)\det(1+L_Bu^{1/2})}{\det(1-L_Eu)\det(1-L_Eu^2)}.$$
\end{theorem}
This gives another interpretation of the zeta identity.
\begin{theorem}[Another zeta identity]\label{newzetaidentity} Let $X_\G = \G \backslash \PGL_3(F)/ \PGL_3(\oo)$. Then
$$\frac{(1-u^3)^{\chi(X_\G)}}{\det(I-A_1u+qA_2u^2-q^3u^3 I)} = Z_1(X_\G,u)Z_-(\G,u).$$
\end{theorem}
The right hand side gives an Euler product expression of the left hand side. Note that $Z_1(X_\G,u)$ is defined geometrically and $Z_-(\G,u)$ algebraically. More precisely,  $Z_1(X_\G,u)$ involves half of straight closed geodesics, namely, those of type $1$, while $Z_-(\G,u)$ involves half of primitive conjugacy classes, namely, those of negative type. For $\PGL_2(F)$,  these two kinds of expressions are equivalent, so that we have both algebraic and geometric interpretations of the Ihara zeta identity for graphs. For $\PGL_3(F)$,
the zeta identity cannot be expressed solely algebraically or geometrically. Indeed it encodes both kinds of information at the same time.

\end{document}